\documentclass[a4paper,12pt,dvipdfmx]{amsart}
\usepackage[hiresbb]{graphicx}
\usepackage{amsmath,amssymb,amsfonts,amsthm,color}
\newtheorem{definition}{Definition}[section]
\newtheorem{theorem}[definition]{Theorem}
\newtheorem{lemma}[definition]{Lemma}

\newtheorem{remark}[definition]{Remark}
\newtheorem{example}[definition]{Example}
\newtheorem{proposition}[definition]{Proposition}

\newcommand{\M}{\mathcal{M}}

\begin{document}

\title[Non-linear traces on semifinite factors]{Non-linear traces on semifinite factors and generalized singular values}
%\author{Masaru Nagisa and Yasuo Watatani}

\author{Masaru Nagisa}
\address[Masaru Nagisa]{Department of Mathematics and Informatics, Faculty of Science, Chiba University, 
Chiba, 263-8522,  Japan: \ Department of Mathematical Sciences, Ritsumeikan University, Kusatsu, Shiga, 525-8577,  Japan}
\email{nagisa@math.s.chiba-u.ac.jp}
\author{Yasuo Watatani}
\address[Yasuo Watatani]{Department of Mathematical Sciences,
Kyushu University, Motooka, Fukuoka, 819-0395, Japan}
\email{watatani@math.kyushu-u.ac.jp}

\maketitle

\begin{abstract}
We introduce non-linear traces of the Choquet type and Sugeno type on a semifinite factor $\M$ as a non-commutative 
analog of the Choquet integral and Sugeno integral for non-additive measures. 
We need weighted dimension function $p \mapsto \alpha(\tau(p))$ for projections $p \in \M$, which is an analog of 
a monotone measure.  
They have certain partial additivities. We show that these partial additivities characterize 
non-linear traces of both the Choquet type and Sugeno type respectively. 
Based on the notion of generalized eigenvalues and singular values, 
we show that non-linear traces of the Choquet type are closely related to the Lorentz function spaces and 
the Lorentz operator spaces if the weight functions $\alpha$ are concave. For the algebras of compact operators 
 and factors of type ${\rm II}$, 
we completely determine the condition that the associated weighted $L^p$-spaces for the non-linear traces become 
quasi-normed spaces in terms of the weight functions $\alpha$ for any $0 < p < \infty$. 
We also show that any non-linear trace of the Sugeno type gives a certain metric on the factor.  
This is an attempt at non-linear and non-commutative integration theory on semifinite factors.



\end{abstract}

\section{Introduction}

We studied several classes of general non-linear positive maps between $C^*$-algebras 
in \cite{nagisawatatani}.  We introduced non-linear traces  
of the Choquet type and the Sugeno type on matrix algebras in \cite{nagisawatatani2} and on the algebra $K(H)$ of 
compact operators in \cite{nagisawatatani3}. 

Ando-Choi \cite{A-C} and Arveson \cite{Ar2} originated the study of  non-linear completely 
positive maps and extended the Stinespring dilation theorem. 
Hiai-Nakamura \cite{H-N} studied a non-linear counterpart of Arveson's 
Hahn-Banach type extension theorem \cite {Ar1} for completely positive linear maps. 
Bel\c{t}it\u{a}-Neeb \cite{B-N} considered non-linear completely positive maps and dilation 
theorems for real involutive algebras.  Dadkhah-Moslehian \cite{D-M} studied some properties of non-linear positive maps 
like Lieb maps and the multiplicative domain for 3-positive maps. Dadkhah-Moslehian-Kian 
\cite{D-M-K} considered the continuity of non-linear positive maps between $C^*$-algebras.

The functional calculus by a continuous positive function is an important example of non-linear positive maps as in  \cite{bhatia1} and \cite{bhatia2}. %, \cite{D} and \cite{Si}. 
In particular, the functional calculus by operator monotone functions is crucial  to study operator means 
in Kubo-Ando theory \cite{kuboando}.  
Another important motivation for the study of non-linear positive maps on $C^*$-algebras is
non-additive measure theory, which was initiated by Sugeno \cite{Su} and  Dobrakov \cite{dobrakov}. 
Choquet integrals \cite{Ch} and Sugeno integrals \cite{Su} are studied as non-linear integrals in 
non-additive measure theory. The differences between them are two operations used: 
sum and product for Choquet integrals and maximum and minimum for Sugeno integrals. 
They have partial additivities. 
Choquet integrals have  comonotonic additivity and 
Sugeno integrals have comonotonic F-additivity (fuzzy additivity). Conversely,  it is known that 
Choquet integrals and Sugeno integrals are characterized by these partial additivities,  see, for example, \cite{De}, 
\cite{schmeidler}, \cite{C-B}, \cite{C-M-R}, \cite{D-G}. 
More precisely, 
comonotonic additivity, positive homogeneity, and monotony characterize 
Choquet integrals. Similarly, comonotonic F-additivity, F-homogeneity, and monotony
characterize Sugeno integrals.
We also note that the inclusion-exclusion integral by Honda-Okazaki \cite{H-O}
for non-additive monotone measures is a general notion of non-linear integrals. 

In this paper, we introduce non-linear traces of the Choquet type and Sugeno type on a semifinite factor $M$ as a non-commutative 
analog of the Choquet integral and Sugeno integral for non-additive measures. 
We need weighted dimension function $p \mapsto \alpha(\tau(p))$ for projections $p \in M$, which is an analog of 
a monotone measure.
They have certain partial additivities. 
We show that these partial additivities characterize 
non-linear traces of both the Choquet type and Sugeno type respectively. 

Lorentz function spaces are typical examples of symmetric function spaces and 
Lorentz operator spaces are typical examples of symmetric operator spaces.  
In \cite{K-S}  using uniform submajorization, Kalton and Sukochev established a 
one-to-one correspondence between the symmetric function spaces and symmetric operator spaces. 
Based on the notion of generalized eigenvalues and singular values, 
we show that non-linear traces of the Choquet type are closely related to the Lorentz function spaces and 
Lorentz operator spaces if the weight functions $\alpha$ are concave.

If $M$ is a factor of type ${\rm I}_{\infty}$, then this corresponds non-linear traces of the Choquet type on the algebra $K(H)$ of compact operators introduced in \cite{nagisawatatani3}. In this article, we broaden it more, and we  
completely determine the condition that the associated weighted $L^p$-spaces for the non-linear traces on 
the algebra $K(H)$ of compact operators become 
quasi-normed spaces in terms of the weight functions $\alpha$ for any $0 < p < \infty$. 
We note that  
Jor-Ting Chan, Chi-Kwong Li, and Charlies C. N. Tu studied some 
classes of these unitarily invariant norms on $B(H)$ from a different view point in \cite{C-L-T}. 
For factors of type ${\rm II}$,  
we also completely determine the condition that the associated  weighted $L^p$-spaces for the non-linear traces become 
quasi-normed spaces in terms of the weight functions $\alpha$ for any $0 < p < \infty$. 

We also show that any non-linear trace of the Sugeno type gives a certain metric on the factor. 
 
This is an attempt at non-linear and non-commutative integration theory on semifinite factors.

\vspace{3mm}

This work was supported by JSPS KAKENHI Grant Numbers\linebreak
JP17K18739 and JP23K03151
and also supported by the Research Institute for Mathematical Sciences,
an International Joint Usage/Research Center located in Kyoto University.

%%%
%%%      Section 2
%%%

\section{Non-linear traces of the Choquet type}
In this section,  we  introduce non-linear traces of the  Choquet type on a semifinite factor $\M$,  
which is  a non-commutative 
analog of the Choquet integrals for non-additive measures. 
We need weighted dimension function $p \mapsto \alpha(\tau(p))$ for projections $p \in \M$, which is an analog of 
a monotone measure.  
If $\M$ is a factor of type ${\rm I}_{\infty}$, then this corresponds non-linear traces of the Choquet type 
on the algebra $K(H)$ of compact operators introduced in \cite{nagisawatatani3}. 
See \cite{bhatia1}, \cite{H}, \cite{hiaipetz}, \cite{Horn-Johnson}, and \cite{simon} on basic facts of eigenvalues and singular values for matrix algebras 
and the algebras of compact operators. 

%%%%%%%%%%%%%%%%%%%%%%%%%%%%%
%     Definition 2.1
%%%%%%%%%%%%%%%%%%%%%%%%%%%%%
\begin{definition} \rm
Let $\Omega$ be a set and ${\mathcal B}$ a ring of sets on $\Omega$, that is, 
${\mathcal B}$ is a family of subsets of $\Omega$ which is closed under the operations 
union $\cup$ and set theoretical difference $\backslash$. Hence  ${\mathcal B}$ is also 
closed under intersection $\cap$. 
A function $\mu: {\mathcal B} \rightarrow [0, \infty]$ is  called a 
{\it monotone measure} if $\mu$ satisfies 
\begin{enumerate}
\item[$(1)$]  $\mu(\emptyset) = 0$, and 
\item[$(2)$] For any $A,B \in {\mathcal B}$, if $A \subset B$, 
then $\mu(A) \leq \mu(B)$. 

\end{enumerate}
\end{definition}

\begin{definition} \rm (Choquet integral)
Let $\Omega$ be a set and ${\mathcal B}$ a ring of sets on $\Omega$.  
Let  $\mu: {\mathcal B} \rightarrow [0, \infty]$ be a monotone measure on $\Omega$. 
Let $f$ be a non-negative measurable function on $\Omega$.  Then the Choquet integral of $f$ is defined by 
\begin{align*}
{\rm (C)} \int f d\mu & := \int_0^{\infty} \mu ( \{ x \in \Omega \ | \ f(x) \geq s \}) ds \\
& =  \int_0^{\infty} \mu ( \{ x \in \Omega \ | \ f(x) > s \}) ds.  
\end{align*}
\end{definition}

\begin{remark} \rm  If $f$ is a simple function with 
$$
f = \sum_{i=1}^{n} a_i \chi _{A_i},  \ \ \  A_i \cap A_j = \emptyset  \  ( i\not=j) 
$$
then the Choquet integral of $f$ is given by 

$$
{\rm (C)}\int f d\mu = \sum_{i=1}^{n-1} (a_{\sigma(i) }- a_{\sigma(i+1)})\mu(A_i) 
+ a_{\sigma(n) }\mu(A_n) , 
$$
where $\sigma$ is a permutation on $\{1,2,\dots n\}$ such that 
$a_{\sigma(1)} \geq a_{\sigma(2)} \geq  \dots \geq a_{\sigma(n)}$ 
and $A_i := \{\sigma(1),\sigma(2),\dots,\sigma(i)\}$. 
\end{remark}

Choquet integral has the following properties: \\
 (1) (monotonicity)  
For any  non-negative measurable functions $f$ and $g$  on $\Omega$, 

$$
0 \leq f \leq g \Rightarrow 0 \leq {\rm (C)} \int f d \mu \leq   {\rm (C)}\int g d \mu .
$$ 
\noindent
(2) (comonotone additivity) 
For any  non-negative measurable functions $f$ and $g$  on $\Omega$, 
if $f$ and $g$ are comonotone, then 
$$
{\rm (C)} \int (f + g)  d \mu
= {\rm (C)} \int f d \mu + {\rm (C)}\int g d \mu ,
$$  
where $f$ and $g$ are {\it comonotone} if 
for any $s,t \in \Omega$ we have  that 
$$
(f(s) -f(t))(g(s)-g(t)) \geq 0, 
$$
that is, 
$$
f(s) < f(t)   \Rightarrow   g(s)  \leq g(t) . 
$$ 

\noindent
(3) (positive homogeneity)
For any  non-negative measurable function $f$ on $\Omega$ and 
any scalar $k \geq 0$, 
$$
 {\rm (C)} \int kf d \mu = k  {\rm (C)} \int f d \mu , 
 $$
where $0 \cdot \infty = 0$.

\begin{remark} \rm 
It is important to note that the Choquet integral can be {\it essentially} characterized as a non-linear monotone positive functional 
which is positively homogeneous and comonotonic additive.  
\end{remark}

\begin{remark} \rm
When $\Omega = {\mathbb R}$, if $f$ and $g$ are monotone increasing, then  $f$ and $g$ are comonotone. 
\end{remark}

For a topological space $X$, let $C(X)$ be the set of all continuous functions on $X$ and 
$C(X)^+$ be the set of all non-negative continuous functions on $X$. More generally, for a $C^*$-algebra $A$, we denote by $A^+$ the set of positive operators of $A$.  

%%%%%%%%%%%%%%%%%%%%%
%% Definition 2.3
%%%%%%%%%%%%%%%%%%%%%
\begin{definition}  \rm Let $A$ be a $C^*$-algebra.  A non-linear positive map 
$\varphi : A^+ \rightarrow  [0, \infty]$ is called a {\it trace} if 
$\varphi$ is unitarily invariant, that is, 
$\varphi(uau^*) = \varphi(a)$ for any $a \in A^+$ and any unitary 
$u$ in the multiplier algebra $M$ of $A$.  
\begin{itemize}
  \item  $\varphi$ is {\it monotone} if 
$a \leq b$ implies  $\varphi(a) \leq  \varphi(b)$ for any $a,b \in  A^+$. 
  \item  $\varphi$ is {\it positively homogeneous} if 
$\varphi(ka) = k\varphi(a)$ for any $a \in A^+$ and any  scalar $k \geq 0$, 
where we regard $0 \cdot \infty = 0$. 
  \item $\varphi$ is {\it comonotonic additive on the spectrum} if 
$$
\varphi(f(a)  + g(a)) = \varphi(f(a)) + \varphi(g(a)) 
$$  
for any $a \in A^+$ and 
any comonotonic functions $f$ and $g$  in $C(\sigma(a))^+$ ( with $f(0) = g(0) = 0$ if $a$ is not invertible), where 
$f(a)$ is a functional calculus of $a$ by $f$ and  $C(\sigma(a))$ is 
the set of continuous functions on the spectrum $\sigma(a)$ of $a$. 
  \item $\varphi$ is {\it monotonic increasing additive on the spectrum} if 
$$
\varphi(f(a)  + g(a)) = \varphi(f(a)) + \varphi(g(a))
$$  
for any $a \in A^+$ and 
any monotone increasing functions $f$ and $g$  in $C(\sigma(a))^+$ (with $f(0)=g(0)=0$ if $a$ is not invertible). 
We may replace $C(\sigma(a))^+$ by $C([0,\infty))^+$, since any monotone increasing function in $C(\sigma(a))^+$ 
can be extended to a monotone increasing function in $C([0,\infty))^+$.  
Then by induction, 
we also have 
$$
 \varphi(\sum_{i=1}^n f_i(a)) = \sum_{i=1}^n \varphi (f_i(a))
% \varphi(f_1)(a) + f_2(a) + \dots +f_n(a)) = \varphi(f_1(a)) + \varphi(f_2(a)) + \dots +\varphi(f_n(a))
$$
for any monotone increasing functions $f_1, f_2, \dots, f_n$ in  $C(\sigma(a))^+$  ( with  $f_i(0) = 0$ for $i=1,2,\ldots, n$ if 
$a$ is not invertible). 
\end{itemize}
\end{definition}

Now we  recall the notions of the generalized singular values (also called generalized $s$-numbers) 
on semifinite von Neumann algebras. 
Murray-von Neumann \cite{murrayvonneumann} started to study the generalized singular values 
for finite factors. The general case of semifinite von Neumann algebras was studied by M. Sonis  \cite{sonis}, 
V. I. Ovchinnikov \cite{ovchinnikov}, F. J. Yeadon \cite{yeason}, and Fack-Kosaki \cite{F-K}.
For the generalized singular values  and the majorization, see also Murray-von Neumann \cite{murrayvonneumann}, 
Kamei \cite{kamei}, Fack \cite{F}, Hiai-Nakamura \cite{H-N2}, Harada-Kosaki \cite{H-K},  and we also recommend a book \cite{L-S-Z} by Lord-Sukochev-Zanin on the Calkin correspondence for semi-finite von Neumann algebras. 
Let $\M$ be a semifinite von Neumann algebra on a Hilbert space $H$ with a faithful normal semifinite trace $\tau$.  A densely defined closed operator 
$a$ affiliated with $\M$ is said to be {\it $\tau$-measurable} if for each $\epsilon > 0$ there exists a projection $p$ in $\M$ such that  $p(H) \subset D(a)$ and $\tau(I-p) < \epsilon$. Let $a = u|a|$ be the polar decomposition of $a$ and 
$|a| = \int_0^{\infty} \lambda d e_{\lambda}$ be the spectral decomposition of $|a|$. Then $a$ is $\tau$-measurable if and only if 
$\lim _{\lambda  \to \infty} \tau(I- e_{\lambda}) = \lim _{\lambda  \to \infty} \tau(e_{({\lambda}, \infty)} )= 0$.  We denote by  $\tilde{\M}$ 
the set of  $\tau$-measurable operators. For a positive self-adjoint operator $a = \int_0^{\infty} \lambda d e_{\lambda}$ 
affiliated with $\M$, we set 
$$
\tau(a) =  \int_0^{\infty} \lambda d\tau( e_{\lambda}). 
$$
 For $p$ with $1 \leq p \leq \infty$, the non-commutative $L^p$-space $L^p(\M;\tau)$   
is the Banach space consisting of all $a \in \tilde{\M}$ with $||a||_p := \tau(|a|^p)^{1/p} < \infty$, see, for example, \cite{nelson}, \cite{segal}, \cite{terp}, and \cite{F-K}. 
Here  $L^{\infty}(\M;\tau) = \M$. 
We denote by $ \tilde{S}$ the closure of $L^1(\M)$ in  $\tilde{\M}$ in the measure topology, where the measure topology on  $\tilde{\M}$ is the topology whose fundamental 
system of neighborhoods around $0$ is given by 
$$
V(\epsilon, \delta) = \{a \in \tilde{\M} | \exists p \in \M : \text{projection s.t.}  
\|ap\| <\epsilon \text{ and }  \tau(I-p) < \delta\}
$$
Then $\M$ is dense in   $\tilde{\M}$ in the measure topology. 

%%%%%%%%%%%%%%%%%%%%%
%% Definition 2.7
%%%%%%%%%%%%%%%%%%%%%
\begin{definition} \rm Let $\M$ be a semifinite von Neumann algebra on a Hilbert space $H$ with a faithful normal semifinite trace $\tau$.  
Let  $a$ be a positive self-adjoint operator affiliated with $\M$ with the spectral decomposition $a =  \int_0^{\infty} \lambda d e_{\lambda}(a)$, 
where $e_t = e_{(-\infty,t]}$ is a spectral projection. 
For $t \geq 0$, the ``generalized $t$-th eigenvalue''
$\lambda_t(a)$ of $a$ is defined by 
$$
\lambda_t(a) := \inf \{s \geq 0 \; |\; \tau(e_{(s, \infty)}(a)) \leq t \} , 
$$
where we put $\inf A = + \infty$ if $A$ is an empty set, by convention.  If $a$ is in $M^+$, then $\lambda_0(a) = \|a \|$. 
Furthermore, if $a$ is a  $\tau$-measurable operator and $t > 0$, then $\lambda_t(a) < \infty$ and it is known that 
$\lambda_t(a)$ is also expressed by 
$$
\lambda_t(a) = \inf \{\|ap\| \ |\; \  p \in \M : \text{ projection s.t. } \tau(I-p) \leq t\}
$$
We also have a min-max type expression 
$$
\lambda_t(a) = \inf \{\sup _{\xi \in pH, \| \xi \| = 1} \langle a\xi,\xi \rangle  \ | \  p \in \M : \text{ projection s.t. } \tau(I-p) \leq t\}
$$
For a general  operator $a$ which is affiliated with $\M$, 
the `` generalized $t$-th singular value''  (i.e. generalized s-number)   
$\mu_t(a)$ of $a$ is defined by 
$\mu_t(a) := \lambda_t(|a|) $.  The map $\mu(a) := [0,\infty) \ni t \mapsto \mu_t(a) \in [0,\infty]$ 
is sometimes called the generalized singular value function of $a$ and satisfies the following properties:
\begin{enumerate}
\item[(1)] The generalized singular value function $\mu(a):t \mapsto \mu_t(a)$ is decreasing and right-continuous. 
\item[(2)] If $0 \leq a \leq b$, then  $\mu_t(a) \leq  \mu_t(b)$. 
\item[(3)] $\mu_{s +t}(a+b) \leq  \mu_s(a)  + \mu_t(b)$. 
\item[(4)] $\mu_{s +t}(ab) \leq  \mu_s(a)  \mu_t(b)$. 
\item[(5)] $\mu_{t}(abc)  \leq \|a\|\mu_{t}(b) \|c\|$.
\item[(6)] $\mu_t(a^*) = \mu_t(a)$. 
\item[(7)] For any unitary $u \in M$, $\mu_t(uau^*) = \mu_t(a)$. 
\item[(8)] $|\mu_t(a) - \mu_t(b)| \leq \|a-b\|$.
\item[(9)] If $a \geq 0$ and $f: [0,\infty) \rightarrow [0,\infty)$ is a continuous and increasing function, then 
$\mu_t(f(a)) = f(\mu_t(a))$. 
\item[(10)] If $\tau (I) = 1$, then $\mu_t(a) = 0$ for all $t >1$. 
\end{enumerate}
See, for example, \cite{F-K}, \cite{fan}, \cite{L-S-Z}, and \cite{ovchinnikov}.

For a self-adjont operator $a \in \M$ with a finite trace $\tau$ ,  ``generalized $t$-th eigenvalue''
$\lambda_t(a)$ of $a$  can be  defined by 
$$
\lambda_t(a) := \inf \{s \in \mathbb R \; |\; \tau(e_{(s, \infty)}(a)) \leq t \} , \ \ (t \in [0, \tau(I)).
$$
See \cite{petz} for example.
But we do not use it in this paper.  For a positive operator $a \in M$,  $\lambda_t(a) = \mu_t(a)$. So 
this does not cause any confusion. 
\end{definition}

\begin{example} \rm If $\M=B(H)$ is a factor of type ${\rm I}_{\infty}$, then $\tilde{\M} = B(H)$ and 
$\tilde{S} = K(H)$.  
\end{example}

\begin{example} \rm Let  $\M=L^{\infty}(\Omega,\mu)$ for some semifinite measure space $(\Omega,\mu)$. 
Then   $\tilde{M}$ consists of measurable functions $f$ on $\Omega$ such that $f$ is bounded except on a set of finite 
measure. In this case, ``generalized $t$-th eigenvalue ``
$\lambda_t(f)$ of positive function $f$ corresponds to the decreasing rearrangement $f^*$ of $f$:
$$
 f^*(t) := \inf \{s \geq 0 \; | \; \mu(\{x  \in \Omega \ | \ |f(x)| > s \}) \leq t \} ,
$$
\end{example}

Now we shall consider non-linear traces of the Choquet type on a semifinite factor $\M$. 
First,  we assume that $\M$ is a factor of type I.  When $\M = M_n({\mathbb C})$, 
we introduced non-linear traces of the Choquet type in \cite{nagisawatatani2}.  When $\M = B(H)$ for a 
separable infinite-dimensional Hilbert space $H$, all proper ideals of $B(H)$ are contained in the ideal $K(H)$ 
of the compact operators. Therefore we studied non-linear traces of the Choquet type on $K(H)$  in \cite{nagisawatatani3}. 
Here we  broaden it more and we  
completely determine the condition that the associated weighted $L^p$-spaces for the non-linear traces on 
$K(H)$ become 
quasi-normed spaces in terms of the weight functions $\alpha$ for any $0 < p < \infty$. 
We recall some notations on non-linear traces of the Choquet type on $K(H)$.  

\begin{definition} \rm
Let $\alpha: {\mathbb N}_0 :=  {\mathbb N} \cup \{0\}  \rightarrow [0, \infty)$ be a 
monotone increasing function with $\alpha(0) = 0$.  
Any $a \in K(H)^+$ can be written as 
$a = \sum_{n=1}^{\infty} \lambda_n(a)p_n$, where each $p_n$ is a one dimensional spectral projection of $a$ 
with  $\sum_{n=1}^{\infty} p_n \leq I$, and 
$$
\lambda(a) := (\lambda_1(a),\lambda_2(a),\ldots, \lambda_n(a), \ldots)
$$
is the list of the 
eigenvalues of $a$ in decreasing order :
$\lambda_1(a) \geq \lambda_2(a) \geq \cdots \geq \lambda_n(a) \geq \cdots $ with 
counting multiplicities, where  the sequence $(\lambda_n(a))_n$ converges to zero.  
In order to make the sequence uniquely determined, we use the following convention: 
If $a \in  K(H)^+$ is not a finite-rank operator, then each $\lambda_n(a) \not= 0$. 
If $a \in  K(H)^+$ is a finite-rank operator, $\lambda(a)$ is given by adding zeroes to the list of eigenvalues of $a$. 
In both cases, the spectrum $\sigma(a)$ of $a$ is given by 
$\sigma(a) = \cup_{n = 1}^{\infty} \{\lambda_n(a)\}   \cup \{0\}$.  We call such 
$\lambda(a)$ the {\it eigenvalue sequence} of $a \in  K(H)^+$. 

\begin{remark} \rm 
In this paper, for a compact positive operator $a \in K(H)^+$, we denote by $\lambda_1(a) = \|a\|$ the 
largest first eigenvalue of $a$.  But for a positive operator $a$ in a factor of Type ${\rm II}$, we denote by 
$\lambda_0(a)= \mu_0(a) =  \|a\|$ the largest generalized 0-th eigenvalue of $a$.  Be careful. 
\end{remark}

We define the non-linear trace $\varphi_{\alpha} : K(H)^+ \rightarrow  [0, \infty]$ of the Choquet type
associated with $\alpha$ by 
$$
\varphi_{\alpha}(a) 
 =  \sum_{i=1} ^{\infty} ( \lambda_i(a)- \lambda_{i+1}(a)) \alpha(i).
$$
Note that $\varphi_{\alpha}$ is lower semi-continuous on $K(H)^+$ in the norm.  

Put 
$c_n = \alpha(n) - \alpha(n-1) \geq 0$ and $c_1 = \alpha(1) - \alpha(0) = \alpha(1)$. 
Then $\alpha(n) = \sum_{i=1}^n c_i$ and 
the non-linear trace $\varphi_{\alpha}$ of the Choquet type associated with $\alpha$ is also 
described as follows: For  $a \in K(H)^+$, 
 $$
 \varphi_{\alpha}(a) := \sum_{i=1} ^{\infty}  \lambda_i(a) (\alpha(i) - \alpha(i-1)) 
 = \sum_{i=1} ^{\infty}  \lambda_i(a) c_i .
 $$
 Conversely for any sequence $c =(c_n)_n$ with each $c_i \in [0,\infty)$, the non-linear trace $\varphi$ 
 defined by 
 $$
  \varphi(a) = \sum_{i=1} ^{\infty}  \lambda_i(a) c_i
 $$
 is a trace of the Choquet type associated with some $\alpha$ such that  $ \alpha(n) = \sum_{i=1}^n c_i$. 

\end{definition}

For $p \geq 1$, we have studied the condition that the weighted Schatten-von Neumann $p$-class  ${\mathcal C}^{\alpha}_p(H)$ for non-linear traces  $\varphi_{\alpha}$ 
of the Choquet type becomes normed spaces in terms of $\alpha$ in \cite{nagisawatatani3}. 

\begin{theorem} \label{thm:norm} (N-W \cite{nagisawatatani3})
Let $\varphi = \varphi_{\alpha}$ be  a non-linear trace of the Choquet type associated with  
a monotone increasing function $\alpha: {\mathbb N}_0  \rightarrow [0, \infty)$  with $\alpha(0) = 0$ and 
$\alpha(1) > 0$.  Fix $p \geq 1$.  
Define $|||a|||_{\alpha,p}:= \varphi_{\alpha}(|a|^p)^{1/p}$ for $a \in K(H)$. 
Then the following conditions are equivalent: 
\begin{enumerate}
\item[$(1)$] $\alpha$ is concave in the sense that 
$\dfrac{\alpha(i +1) + \alpha(i - 1)}{2} \leq \alpha(i), \;  (i =1,2,3, $ $\dots)$.
\item[$(2)$] $(c_i)_i$ is a decreasing sequence: $c_1 \geq c_2 \geq \cdots \geq c_n \geq \cdots $.
\item[$(3)$] $||| \ |||_{\alpha,p}$ satisfies the triangle inequality: for any $a,b \in K(H) $, 
$|||a + b|||_{\alpha,p} \leq |||a |||_{\alpha,p} + ||| b|||_{\alpha,p}$ admitting $+\infty$. 
\end{enumerate}
\end{theorem}

\begin{remark} \rm
An operator $a \in K(H)$ is said to be a {\it weighted Schatten-von Neumann $p$-class operator for a non-linear trace}  $\varphi_{\alpha}$  {\it of the Choquet type} if 
$|||a|||_{\alpha,p}= \varphi_{\alpha}(|a|^p)^{1/p}< \infty$.  The weighted Schatten-von Neumann $p$-class  ${\mathcal C}^{\alpha}_p(H)$ is defined as 
the set of all weighted Schatten-von Neumann $p$-class  operators for a non-linear trace $\varphi_{\alpha}$. We also call ${\mathcal C}^{\alpha}_p(H)$
the weighted $L^p$-space for $\varphi_{\alpha}$.   
Assume that  $\alpha$ is concave.
Then the weighted $L^p$-space ${\mathcal C}^{\alpha}_p(H)$ 
for a non-linear trace $\varphi_{\alpha}$ 
is a Banach space with respect to the norm $||| \ |||_{\alpha,p}$.  Moreover  ${\mathcal C}^{\alpha}_p(H)$ 
is a $*$-ideal of $B(H)$.   
\end{remark}

Now we  shall 
completely determine the condition that the associated weighted $L^p$-spaces for the non-linear traces on $K(H)$ become 
quasi-normed spaces in terms of the weight functions $\alpha$ for any $0 < p < \infty$. We recall the definition of 
quasi-norms (see, for example, \cite{wilansky}). 

%%%%%%%%%%%%%%%%%%%%%%%%%%%%%%
%%%%%  definition 2.13
%%%%%%%%%%%%%%%%%%%%%%%%%%%%%%
\begin{definition} \rm %(quasi-norm, see, for example, \cite{wilansky})
Let $X$ be a complex vector space.
A function $\varphi$ on $X$ is called a quasi-norm if there exists some constant $c>1$ such that for any  $x,y\in X$ and $\alpha\in \mathbb{C}$, $\varphi$ satisfies the following: 
\begin{enumerate}
  \item[(1)] $\varphi(x) \ge 0$.
  \item[(2)] $\varphi(\alpha x)=|\alpha|\varphi(x)$.
  \item[(3)] $\varphi(x+y) \le c (\varphi(x) + \varphi(y))$.
\end{enumerate}
A typical example of quasi-norms is the function $\varphi(x)=(\sum_{i=1}^n |x_i|^p)^{1/p}$ on $X=\mathbb{C}^n$.
When $0<p<1$, $\varphi$ is not a norm on $X$ and satisfies the following relation:
\[   \varphi(x+y) \le 2^{\frac{1}{p}-1}(\varphi(x)+\varphi(y) ).   \]
\end{definition}

\begin{lemma} [Weyl's inequality, see \cite{bhatia1} for matrices and \cite{fan} for compact operators]
For any positive compact operators $a,b \in K(H)^+$, we have
\[   \lambda_{i+j-1}(a+b) \le \lambda_i(a) + \lambda_j(b). \]
\end{lemma}

For $a\in K(H)$, we use the notation $s_i(a)$ which means the $i$-th singular value of $a$ and $s_i(a)=\lambda_i(|a|)$. 
Since  $s_{i+j-1}(a+b) \le s_i(a) + s_j(b)$, we remark that
\[  s_{2i}(a+b) \le s_i(a)+s_{i+1}(b)\le s_i(a)+s_i(b), \text{ and } s_{2i-1}(a+b) \le   s_i(a)+s_i(b).  \]

In order to determine the condition on $\alpha$ when $\varphi_{\alpha}(|a|^p)^{1/p}$ becomes a quasi-norm, 
it is sufficient to consider it on the set $F(H)$ of finite rank operators. 
%%%%%%%%%%%%%%%%%%%%%%%%%%%%%%%%%%%%%%%%%%%%
%      Theorem 2.15
%%%%%%%%%%%%%%%%%%%%%%%%%%%%%%%%%%%%%%%%%%%%
\begin{theorem}
\label{thm:qnorm}
Let $\varphi = \varphi_{\alpha}$ be  a non-linear trace of the Choquet type associated with  
a monotone increasing function $\alpha: {\mathbb N}_0  \rightarrow [0, \infty)$  with $\alpha(0) = 0$ and 
$\alpha(1) > 0$.  For $p$ with  $0 < p < \infty$.  
Define $|||a|||_{\alpha,p}:= \varphi_{\alpha}(|a|^p)^{1/p}$ for $a \in F(H)$. 
Then the following conditions are equivalent: 
\begin{enumerate}
  \item[$(1)$]  There exists a positive number $M$ such that
\[  \varphi_\alpha(|a+b|) \le M (\varphi_\alpha(|a|) + \varphi(|b|) ) \text{ for any } a,b\in F(H).  \]
  \item[$(2)$]  The sequence $(\dfrac{\alpha(2n)}{\alpha(n)})_{n \ge 1}$ is bounded.
  \item[$(3)$]  $||| \cdot |||_{\alpha,p}$ is a quasi-norm on $F(H)$ for any $p\ge1$.
  \item[$(4)$]  $||| \cdot |||_{\alpha,p}$ is a quasi-norm on $F(H)$ for any $0<p<1$.
  \item[$(5)$]  $||| \cdot |||_{\alpha,p}$ is a quasi-norm on $F(H)$ for some $p>0$.
%  \item[$(5)$]  $||| \cdot |||_{\alpha,p}$ is a quasi-norm on $F(H)$ for some $p\ge1$.
%  \item[$(6)$]  $||| \cdot |||_{\alpha,p}$ is a quasi-norm on $F(H)$ for some $0<p<1$.
\end{enumerate}
\end{theorem}
\begin{proof}
$(1)\Rightarrow(5)$ The condition (1) means that $|||\cdot|||_{\alpha,1}$ is a quasi-norm on $F(H)$.

$(5)\Rightarrow(2)$  Let $p>0$ and 
\[   \varphi_\alpha(|a+b|^p)^{1/p} \le L (\varphi_\alpha(|a|^p)^{1/p} + \varphi_\alpha(|b|^p)^{1/p} ).  \]
Choose  $e$ and $f$ are relatively orthogonal projections with dimension $n$ and 
put $a=e$ and $b=f$.
Then $\varphi_\alpha(|a+b|^p) = \varphi_\alpha(e+f) = \sum_{i=1}^{2n}c(i) = \alpha(2n)$ and
$\varphi_\alpha(|a|^p) = \varphi_\alpha(e)=\sum_{i=1}^n c(i) =\alpha(n)$.
Similarly $ \varphi_\alpha (|b|^p) = \alpha (n)$.
By the assumption we have  $\alpha(2n)^{1/p} \le 2L\alpha(n)^{1/p}$ for any $n$.
So $\frac{\alpha(2n)}{\alpha(n)} \le (2L)^p$.
%$(1)\Rightarrow(2)$ Let $e$ and $f$ be relatively orthogonal projections with dimension $n$.
%Put $a=e$ and $b=f$.
%Then $\varphi_\alpha(|a+b|) = \varphi_\alpha(e+f) = \sum_{i=1}^{2n}c(i) = \alpha(2n)$ and
%$\varphi_\alpha(|a|) = \varphi_\alpha(e)=\sum_{i=1}^n c(i) =\alpha(n)$.
%Similarly $ \varphi_\alpha (|b|) = \alpha (n)$.
%So we have
%\[   \frac{\alpha(2n)}{\alpha(n)} = 2 \frac{ \varphi_\alpha(|a+b|)}{\varphi_\alpha(|a|)+\varphi_\alpha(|b|)} \le 2M  \]
%for any $n$,

$(2)\Rightarrow(3)$ By the assumption, we may choose a positive number $L$ such that $\alpha(2n)\le L\alpha(n)$ for all $n$.
This means that $\sum_{i=1}^{2n}c(i) \le L \sum_{i=1}^n c(i)$ for all $n$.
We remark that
\[   (\sum_{i=1}^\infty |x_i+y_i|^p)^{1/p} \le (\sum_{i=1}^\infty |x_i|^p)^{1/p} +  (\sum_{i=1}^\infty |y_i|^p)^{1/p}  \]
for $x_i,y_i\in \mathbb{C}$, since $p\ge1$.

\begin{align*}
    &  \varphi_\alpha(|a+b|^p)^{1/p} = \Bigl( \sum_{i=1}^\infty c(i)s_i(a+b)^p\Bigr)^{1/p}  \\
  = & \Bigl( \sum_{i=1}^\infty c(2i-1)s_{2i-1}(a+b)^p + \sum_{i=1}^\infty c(2i)s_{2i}(a+b)^p \Bigr)^{1/p}  \\
%  = & ( \sum_{i=1}^\infty (c(2i-1)^{1/p}s_{2i-1}(a+b))^p + \sum_{i=1}^\infty (c(2i)^{1/p}s_{2i}(a+b))^p )^{1/p}  \\ 
% \le & ( \sum_{i=1}^\infty (c(2i-1)^{1/p}(s_i(a) +s_i(b)))^p + \sum_{i=1}^\infty (c(2i)^{1/p}(s_i(a) + s_i(b)))^p )^{1/p}  \\
  \le &  \Bigl( \sum_{i=1}^\infty c(2i-1)(s_i(a) +s_i(b))^p + \sum_{i=1}^\infty c(2i)(s_i(a) + s_i(b))^p \Bigr)^{1/p}  \\
  = & \Bigl( \sum_{i=1}^\infty (c(2i-1)+c(2i))(s_i(a) +s_i(b))^p \Bigr)^{1/p}  \\
  = &  \Biggl( \sum_{i=1}^\infty \Bigl( (c(2i-1)+c(2i))^{1/p}(s_i(a) +s_i(b)) \Bigr)^p \Biggr)^{1/p}  \\
 \le &  \Biggl( \sum_{i=1}^\infty \Bigl( (c(2i-1)+c(2i))^{1/p}s_i(a) \Bigr)^p \Biggr)^{1/p} 
       + \Biggl( \sum_{i=1}^\infty \Bigl( (c(2i-1)+c(2i) )^{1/p}s_i(b)\Bigr)^p \Biggr)^{1/p} ,
\end{align*}
where we apply the $\ell^p$-subadditivity for elements
$( (c(2i-1)+c(2i))^{1/p}s_i(a) )_i$ and  $( (c(2i-1)+c(2i))^{1/p}s_i(b))_i$ in $\ell^p$. 
%\[ ( (c(2i-1)+c(2i))^{1/p}s_i(a) )_i,  ( (c(2i-1)+c(2i))^{1/p}s_i(b))_i \in \ell^p.  \]
Then the first term  can be computed as follows:
\begin{align*}
    &   \Biggl( \sum_{i=1}^\infty \Big((c(2i-1)+c(2i))^{1/p}s_i(a)\Bigl)^p \Biggr)^{1/p}  
  =   \Bigl( \sum_{i=1}^\infty (c(2i-1)+c(2i))(s_i(a))^p \Bigr)^{1/p}  \\
  = &  \Bigl( (c(1)+c(2))s_1(a)^p + (c(3)+c(4))s_2(a)^p + (c(5)+c(6))s_3(a)^p + \cdots \Bigr)^{1/p} \\
  = &  \Bigl( (c(1)+c(2))(s_1(a)^p-s_2(a)^p) \\
    &  \qquad \qquad + (c(1)+c(2)+c(3)+c(4))(s_2(a)^p-s_3(a)^p)  \\
    &  \qquad \qquad + (c(1)+c(2)+c(3)+c(4)+c(5)+c(6))(s_3(a)^p-s_4(a)^p) + \cdots \Bigr)^{1/p}  \\
  \le &  \Bigl( L c(1)(s_1(a)^p-s_2(a)^p) + L (c(1)+c(2))(s_2(a)^p-s_3(a)^p)  \\
    & \qquad \qquad + L (c(1)+c(2)+c(3)) (s_3(a)-s_4(a))^p + \cdots \Bigr)^{1/p}  \\
  = & L^{1/p} \Bigl( c(1)s_1(a)^p + c(2)s_2(a)^p + c(3) s_3(a)^p + \cdots \Bigr)^{1/p} = L^{1/p} \varphi_\alpha(|a|^p)^{1/p} .
\end{align*}
By a similar calculation, we can get
\[  \Biggl( \sum_{i=1}^\infty \Bigl( (c(2i-1)+c(2i))^{1/p}s_i(b)\Bigr)^p \Biggr)^{1/p} \le L^{1/p} \varphi_\alpha(|b|^p)^{1/p} .  \]
So we have $\varphi_\alpha(|a+b|^p)^{1/p} \le L^{1/p} \Bigl(\varphi_\alpha(|a|^p)^{1/p} + \varphi_\alpha(|b|^p)^{1/p} \Bigr)$.

In particular, we remark that the implication $(2)\Rightarrow(1)$ is also proved.
So we have $(1)\Leftrightarrow(5)\Leftrightarrow(2)\Rightarrow(3)$.

$(2)\Rightarrow(4)$ We also assume that $\alpha(2n)\le L\alpha(n)$ for all $n$.
This means that $\sum_{i=1}^{2n}c(i) \le L \sum_{i=1}^n c(i)$ for all $n$.
We also remark that 
\[   s_{2i-1}(a+b)^p \le s_i(a)^p+s_i(b)^p  \text{ and } s_{2i}(a+b)^p \le s_i(a)^p+s_i(b)^p  \] 
since $x^p\le y^p+z^p$ for any $x,y,z \ge0$ with $x\le y+z$ and $0<p<1$.

\begin{align*}
    &  \varphi_\alpha(|a+b|^p)^{1/p} = \Bigl(\sum_{i=1}^\infty c(i)s_i(a+b)^p\Bigr)^{1/p}  \\
%  = & ( \sum_{i=1}^\infty c(2i-1)s_{2i-1}(a+b)^p + \sum_{i=1}^\infty c(2i)s_{2i}(a+b)^p )^{1/p}  \\
%  \le & ( \sum_{i=1}^\infty c(2i-1)(s_i(a)^p+s_i(b)^p) + \sum_{i=1}^\infty c(2i)(s_i(a)^p +s_i(b)^p) )^{1/p}  \\
  \le & \Bigl( \sum_{i=1}^\infty (c(2i-1)+c(2i)) (s_i(a)^p+s_i(b)^p) \Bigr)^{1/p} \\
  = & \Biggl( \sum_{i=1}^\infty \Bigl( (c(2i-1)+c(2i))^{1/p} s_i(a) \Bigr)^p + \sum_{i=1}^\infty \Bigl( (c(2i-1)+c(2i))^{1/p} s_i(b) \Bigr)^p \Biggr)^{1/p} \\
  \le & 2^{\frac{1}{p}-1} \Biggl(  \Bigl(\sum_{i=1}^\infty ( (c(2i-1)+c(2i))^{1/p} s_i(a) \bigr)^p\Bigr)^{1/p} \\
       & \qquad \qquad \qquad + (\sum_{i=1}^\infty ( (c(2i-1)+c(2i))^{1/p} s_i(b) )^p )^{1/p}  \Biggr) \\
\intertext{(where we use the property of quasi-norm  $\|\cdot\|_p$ for $0<p<1.$) }
  = & 2^{\frac{1}{p}-1} \Biggl(  \Bigl( \sum_{i=1}^\infty  (c(2i-1)+c(2i)) s_i(a)^p \Bigr)^{1/p} +   \Bigl( \sum_{i=1}^\infty  (c(2i-1)+c(2i)) s_i(b)^p \Bigr)^{1/p}  \Biggr).
\end{align*}
Using the same argument in $(2)\Rightarrow(3)$,
\begin{gather*}
    \Bigl(\sum_{i=1}^\infty  (c(2i-1)+c(2i)) s_i(a)^p \Bigr)^{1/p}  \le L^{1/p} \varphi_\alpha(|a|^p)^{1/p}  \\
    \Bigl(\sum_{i=1}^\infty  (c(2i-1)+c(2i)) s_i(b)^p \Bigr)^{1/p}  \le L^{1/p} \varphi_\alpha(|b|^p)^{1/p} .
\end{gather*} 
So we have $\varphi_\alpha(|a+b|^p)^{1/p} \le 2^{\frac{1}{p}-1} L^{1/p} \Bigl( \varphi_\alpha(|a|^p)^{1/p} + \varphi_\alpha(|b|^p)^{1/p} \Bigr)$.

$(3)\Rightarrow(5)$ and  $(4)\Rightarrow(5)$ are clear. 
%$(5)$ or $(6)\Rightarrow(2)$  Let $p>0$ and 
%\[   \varphi_\alpha(|a+b|^p)^{1/p} \le L (\varphi_\alpha(|a|^p)^{1/p} + \varphi_\alpha(|b|^p)^{1/p} ).  \]
%Choose $a=e$ and $b=f$ as in $(1)\Rightarrow(2)$.
%Then we have $\alpha(2n)^{1/p} \le L(2\alpha(n))^{1/p}$, that is,  $\frac{\alpha(2n)}{\alpha(n)} \le 2L^p$.
\end{proof}

%%%%%%%%%%%%%%%%%
%  Remark 2.17
%%%%%%%%%%%%%%%%%
\begin{remark} \rm
Assume that the sequence $(\frac{\alpha(2n)}{\alpha(n)})_n$ is bounded. For any $0 < p < \infty$, we can extend 
$|||a|||_{\alpha,p}$ from $a \in F(H)$ to $a \in K(H)$ as follows: Since
$\lambda_i : K(H)^+ \ni a \mapsto \lambda_i(a) \in [0,\infty)$ 
is operator norm continuous, for a fixed $N$, 
$K(H)^+ \ni a \mapsto (\sum_{i=1} ^{N}  \lambda_i(a^p) c_i )^{1/p}$ is also norm continuous. Hence 
we can extend $|||a|||_{\alpha,p}$ from $a \in F(H)^+$ to $a \in K(H)^+$ admitting $+\infty$ by 
$$
K(H)^+ \ni a \mapsto |||a|||_{\alpha,p} :=(\sum_{i=1} ^{\infty}  
 \lambda_i(a^p) c_i )^{1/p} 
= \sup_N (\sum_{i=1} ^{N}  \lambda_i(a^p) c_i )^{1/p}
$$
as a lower semi-continuous function with respect to the operator norm. 
Since a continuous function $[0,\infty) \ni x \mapsto x^p$ is monotone increasing, 
$$
\lambda_i(a^p) = \lambda_i(a)^p.
$$
For any projections $e$ and $f$ with $0 \leq e \leq f$, 
$$
\lambda_i(eae) = \lambda_i(efafe) \leq \|e\|  \lambda_i(faf) \|e\| \leq  \lambda_i(faf).
$$
Hence $|||eae|||_{\alpha,p}  \leq  |||faf||_{\alpha,p}$. 
Take $a \in K(H)^+$. Let $(e_n)_n$ be an increasing sequence of projections with $\dim e_n = n$.  
Since $a$ is a compact operator,  $e_nae_n$ converges to $a$ in the operator norm. 
Since $ K(H)^+ \ni a \mapsto |||a|||_{\alpha,p}$ is  lower semi-continuous function 
with respect to the operator norm,  $|||e_nae_n|||_{\alpha,p}$ converges to $|||a|||_{\alpha,p}$.  
By the above Theorem, 
there exists a constant $M > 0$ such that for any  $a,b \in K(H)^+$ and for any $n$ 
$$
|||e_n(a + b)e_n|||_{\alpha,p} 
\leq M( |||e_nae_n|||_{\alpha,p} +   |||e_nbe_n|||_{\alpha,p}) .
$$
Taking $n \to \infty$, we obtain that 
$$
|||a + b|||_{\alpha,p} 
\leq M( |||a|||_{\alpha,p} +   |||b|||_{\alpha,p}). 
$$
Next, we consider the general case such that $a \in  K(H)$. 
Define 
$$
|||a|||_{\alpha,p} := |||(|a|)|||_{\alpha,p} = (\sum_{i=1} ^{\infty}  
 \lambda_i(|a|^p) c_i )^{1/p}. 
$$ 
For not necessarily positive $a,b \in K(H)$, use a known technique. 
By a theorem of Akemann-Anderson-Pedersen \cite{A-A-P}, which  generalizes Thompson's Theorem \cite{To}, 
there exist isometries $u,v \in B(H)$ such that 
$$
| a + b | \leq u|a|u^* + v|b|v^*.
$$
Since  $\lambda_i((u|a|u^*)^p) = \lambda_i(|a|^p)$,  we have that 
$$
|||a + b|||_{\alpha,p} 
\leq M( |||a|||_{\alpha,p} +   |||b|||_{\alpha,p}) .
$$
\end{remark}

%%%%%%%%%%%%%%%%%%%%
%  Example 2.18
%%%%%%%%%%%%%%%%%%%
\begin{example} \rm
 Put $\alpha(0) = 0, \alpha(1) = \alpha(2) = 1$, and $\alpha(n) = 3$ for $n \ge 3$. 
Then for a positive $a \in K(H)$, 
$$
\varphi_{\alpha}(a) = \lambda_1(a) + 2\lambda_3(a). 
$$
Since $\alpha$ is not concave and  $\frac{\alpha(2n)}{\alpha(n)} \le 3$, 
$\varphi_{\alpha}(|a|)$ does not give a norm but give a quasi-norm. 
\end{example}

%%%%%%%%%%%%%%%%%%%%%%%%%%%%%%%
%     Definition 2.19
%%%%%%%%%%%%%%%%%%%%%%%%%%%%%%%
\begin{definition} \rm 
Consider any $p$ with $0 < p < +\infty$. An operator $a \in K(H)$ is said to be a {\it weighted Schatten-von Neumann $p$-class operator for a non-linear trace}  $\varphi_{\alpha}$  {\it of the Choquet type} if 
$|||a|||_{\alpha,p}= \varphi_{\alpha}(|a|^p)^{1/p}< \infty$.  The weighted Schatten-von Neumann $p$-class  ${\mathcal C}^{\alpha}_p(H)$ is defined as 
the set of all weighted Schatten-von Neumann $p$-class  operators for a non-linear trace $\varphi_{\alpha}$. We also call ${\mathcal C}^{\alpha}_p(H)$
the weighted $L^p$-space for $\varphi_{\alpha}$.   
Assume that  the sequence $(\frac{\alpha(2n)}{\alpha(n)})_n$ is bounded.
Then the weighted $L^p$-space ${\mathcal C}^{\alpha}_p(H)$ 
for a non-linear trace $\varphi_{\alpha}$ 
is a quasi-normed space with respect to the quasi-norm $||| \ |||_{\alpha,p}$.  
Therefore by Aoki-Rolewiz Theorem, ${\mathcal C}^{\alpha}_p(H)$ can be a metric space.  Moreover 
${\mathcal C}^{\alpha}_p(H)$ is a $*$-ideal of $B(H)$.   
\end{definition}

In the rest of this section, we assume that $\M$ is a factor of type ${\rm II}_{\infty}$ with a faithful normal semifinite trace $\tau$ 
or a factor of type ${\rm II}_1$ with a faithful normal finite trace $\tau$ with $\tau(I) = 1.$ 

%%%%%%%%%%%%%%%%%%%%%%%%%%%%%%%
%     Definition 2.20
%%%%%%%%%%%%%%%%%%%%%%%%%%%%%%%
\begin{definition} \rm Let $\M$ be a factor of  type ${\rm II}_{\infty}$.
Let $\alpha: [0,\infty)  \rightarrow [0, \infty)$ be a 
monotone increasing function with $\alpha(0) = 0$.  Assume that $\alpha$ is left continuous. 
We think a weighted dimension function $p \mapsto \alpha(\tau(p))$ for projections $p \in \M$ is an analog of 
a monotone measure.  
Define a non-linear trace $\varphi_{\alpha} : \M^+ \rightarrow  [0, \infty]$ of the Choquet type associated with $\alpha$ 
as follows: For $a \in \M^+$ with the spectral decomposition $a = \int_0^{\infty} \lambda d e_{\lambda}(a)$,  let  
$$
\varphi_{\alpha}(a) : = \int_0^{\infty}  \alpha(\tau(e_{(s,\infty)}(a)))ds
= \int_0^{\infty}  \alpha(\tau({\rm supp}((a-sI)_+)))ds, 
$$
where $(a-sI)_+: = ( (a-sI) + |a-sI|)/2$ is the positive part of  $a-sI$ and ${\rm supp}((a-sI)_+)$ is the 
support projection of $(a-sI)_+$, which is also the range projection of it. 
Here we use the following convention $\alpha(\infty) = \lim_{x \to \infty} \alpha(x)$ to allow $\infty$. 
Then for any unitary $u \in \M$, 
$$
\varphi_{\alpha}(uau^*) =  \varphi_{\alpha}(a),  
$$
because $\tau(e_{(s,\infty)}(uau^*))) = \tau(ue_{(s,\infty)}(a))u^*) =  \tau(e_{(s,\infty)}(a))$.  Hence $\varphi_{\alpha}$ is unitarily invariant.

If $\M$ is a factor of  type ${\rm II}_1$  with a faithful normal finite trace $\tau$ with $\tau(I) = 1$, 
then we consider a monotone increasing function $\alpha: [0,1] \rightarrow [0, \infty)$ 
with $\alpha(0) = 0$.  Assume that $\alpha$ is left continuous. Then we can 
define a non-linear trace $\varphi_{\alpha} : \M^+ \rightarrow  [0, \infty]$ of the Choquet type associated with $\alpha$ 
similarly.

For example,  if $a$ has a finite spectrum and 
$a = \sum_{i=1}^{n} a_ip_i$ is the spectral decomposition with $a_1 \geq a_2 \geq \dots \geq a_n > 0$, then 
$$
\varphi_{\alpha}(a) =  \sum_{i=1} ^{n-1} (a_i- a_{i+1})\alpha(\tau(p_1 + \dots + p_i)) +
 a_n \alpha(\tau(p_1 + \dots + p_n)). 
$$
\end{definition}

%%%%%%%%%%%%%%%%%%%%%%%%%%%%%%%
%     Remark 2.21
%%%%%%%%%%%%%%%%%%%%%%%%%%%%%%%
\begin{remark}\rm
For  $a \in \M^+$, 
$$
 \varphi_{\alpha}(a) := \int_0^{\infty}  \alpha(\tau(e_{(s,\infty)}(a)))ds  
 =  \int_0^{\infty}  \alpha(\tau(e_{[s,\infty)}(a)))ds .
$$
Since one point set $\{0\}$ is a null set, 
it is enough to show that 
$$
\int_{(0,\infty)}  \alpha(\tau(e_{(s,\infty)}(a)))ds  
 =  \int_{(0,\infty)}  \alpha(\tau(e_{[s,\infty)}(a)))ds. 
$$
Put 
$$
A :=\int_{(0,\infty)}  \alpha(\tau(e_{(s,\infty)}(a)))ds , \ 
A_n :=\int_{(\frac{1}{n},\infty)}  \alpha(\tau(e_{(s,\infty)}(a)))ds , 
$$
and 
$$
B :=  \int_{(0,\infty)}  \alpha(\tau(e_{[s,\infty)}(a)))ds, \ 
B_n := \int_{(\frac{1}{n},\infty)} \alpha(\tau(e_{[s,\infty)}(a)))ds. 
$$
Since $A_n \leq B_n$ and $A_{n+1} \geq B_n$, 
taking their limits as $n \to \infty$, we have that $A = B$  by monotone convergence theorem. 
\end{remark}

%%%%%%%%%%%%%%%%%%%%%%%%%%%%%%%
%     Lemma 2.22
%%%%%%%%%%%%%%%%%%%%%%%%%%%%%%%
\begin{lemma}
\label{prop:observation}
Let $\M$ be  a factor of  type ${\rm II}_{\infty}$.
Let $\alpha: [0,\infty)  \rightarrow [0, \infty)$ be a 
monotone increasing left  continuous function with $\alpha(0) = 0$.  
Let  $\varphi_{\alpha} : \M^+ \rightarrow  [0, \infty]$ be a non-linear trace of the Choquet type associated with $\alpha$. 
Then $\varphi_{\alpha}$ is a monotone map and sequentially normal in the sense that for any increasing  sequence 
$(a_n)_n$ in $\M^+$ if $ a_n \nearrow a \in \M^+$ in strong operator topology  then $\varphi_{\alpha}(a_n)  \nearrow \varphi_{\alpha}(a) $.

If $\M$ is  a factor of  type ${\rm II}_1$  with a faithful normal finite trace $\tau$ with $\tau(I) = 1$,
let $\alpha: [0,1]  \rightarrow [0, \infty)$ be a 
monotone increasing left  continuous function with $\alpha(0) = 0$.
Then similar statements hold.   
\end{lemma}
\begin{proof} Let $ 0 \leq a \leq b$ in $\M$.  
For $s \geq 0$ and a natural number $n$, define functions $h_s$ and  $h_s^n$ on $[0,\infty)$ by 
$$
h_s(x) = \chi_{(s,\infty)}(x) = 
\begin{cases}
0, & (0 \leq x \leq s), \\
1, & (s <  x), 
\end{cases}
$$
$$
h_s^n(x) = 
\begin{cases}
0, & (0 \leq x \leq s), \\
nx-ns, & (s \leq x \leq s + 1/n), \\
1, & (s + 1/n \leq  x),
\end{cases}
$$
Then increasing sequence $(h_s^n)_n$ converges to $h_s$ pointwisely.  Hence Borel functional calculus 
$(h_s^n(a))_n$ converges to $h_s(a)$ in the strong operator topology.  Since $\tau$ is normal, 
$(\tau(h_s^n(a))_n$ converges to $\tau(h_s(a))$.  Although 
$ 0 \leq a \leq b$  does not imply that $ h_s(a)\leq  h_s(b)$, it holds that 
$\tau(f(a))\leq \tau(f(b))$ for any continuous increasing function $f$ on $[0,\infty)$ with $f(0) = 0$ 
as in \cite[Corollary 2.9]{F-K}.  Therefore  $\tau(h_s^n(a)) \leq  \tau(h_s^n(b))$. This implies that 
$\tau(h_s(a)) \leq \tau(h_s(b))$.  Since 
$$
h_s(a) = {\rm supp}((a-sI)_+) = e_{(s,\infty)}(a), 
$$
we should note that 
\begin{align*}
\varphi_{\alpha}(a) 
& = \int_0^{\infty}  \alpha(\tau({\rm supp}((a-sI)_+)))ds  \\
&= \int_0^{\infty}  \alpha(\tau(h_s(a))))
= \int_0^{\infty}  \alpha(\tau(e_{(s,\infty)}(a)))ds . 
\end{align*}
Because  $\alpha$ is increasing, 
\begin{align*}
\varphi_{\alpha}(a) 
 & = \int_0^{\infty}  \alpha(\tau((h_s(a)))ds \\
  & \leq \int_0^{\infty}  \alpha(\tau((h_s(b)))ds = \varphi_{\alpha}(b). 
\end{align*}
Thus  $\varphi_{\alpha}$ is a monotone map. 

Assume that an increasing  sequence 
$(a_n)_n$ in $M^+$ converges to $a \in \M^+$  in strong operator topology.  
Then by the monotone convergence theorem, 
\begin{align*}
   &  \lim_{n \to \infty} \varphi_{\alpha}(a_n) 
   = \lim_{n \to \infty} \int_0^{\infty}  \alpha(\tau(h_s(a_n)))ds 
   = \int_0^{\infty} \lim_{n \to \infty} \alpha(\tau(h_s(a_n)))ds \\
 = &  \int_0^{\infty} \sup_n \sup_m \alpha(\tau(h_s^m(a_n)))ds 
 =  \int_0^{\infty} \sup_m \sup_n \alpha(\tau(h_s^m(a_n)))ds \\
 = & \int_0^{\infty} \sup_m \alpha(\tau(h_s^m(a)))ds 
 = \int_0^{\infty}  \alpha(\tau(h_s(a)))ds  = \varphi_{\alpha}(a), 
\end{align*}
since $\alpha$ is left continuous. 
Thus $\varphi_{\alpha}$ is sequentially normal.
\end{proof}

%%%%%%%%%%%%%%%%%%%%%%%%%%%%%%%
%     Definition 2.23
%%%%%%%%%%%%%%%%%%%%%%%%%%%%%%%
\begin{definition} \rm
We say that $a \in \M^+$ is of  {\it $\tau$-finite rank}  in $M$ if $a$ has a finite spectrum and $\tau(a) < \infty.$  It is easy to see 
that $a \in \M^+$ is of $\tau$-finite rank in $\M$  if and only if  $a$ has the spectral decomposition  
 $a = \sum_{k=1}^{n} a_k p_k$ such that each $\tau(p_k) < \infty$ for 
$k = 1,2,\dots,n$.  We denote by $M^+_{\tau-fr}$ the set of all $a \in \M^+$ which is of $\tau$-finite rank in $\M$. 
\end{definition}

We can express non-linear traces of the Choquet type using the ``generalized $t$-th eigenvalues''. 
%%%%%%%%%%%%%%%%%%%%%%%%%%%%%%%
%     Proposition 2.24
%%%%%%%%%%%%%%%%%%%%%%%%%%%%%%%
\begin{proposition}
\label{prop:Stieltjes}
Let $\M$ be a factor of type ${\rm II}_{\infty}$. 
Let $\alpha: [0,\infty)  \rightarrow [0, \infty)$ be a 
monotone increasing left continuous function with $\alpha(0) = 0$.  
Let  $\varphi_{\alpha} : \M^+ \rightarrow  [0, \infty]$ be a non-linear trace of the Choquet type associated with $\alpha$ . 
Then the non-linear trace $\varphi_{\alpha}$ is also described as follows: Let $\nu_{\alpha}$ be the Lebesgue-Stieltjes measure 
associated with $\alpha$ such that $\nu_{\alpha}([a,b)) = \alpha(b) - \alpha(a)$ .
For any $a \in \M^+$ we have that 
$$
\varphi_{\alpha}(a) 
 := \int_0^{\infty}  \alpha(\tau(e_{(s,\infty)}(a)))ds = \int_0^{\infty} \lambda_t(a) d\nu_{\alpha}(t). 
$$
where we admit $+\infty$.

If $\M$ is  a factor of  type ${\rm II}_1$  with a faithful normal finite trace $\tau$ with $\tau(I) = 1$,
let $\alpha: [0,1]  \rightarrow [0, \infty)$ be a 
monotone increasing left  continuous function with $\alpha(0) = 0$.
Then a similar fact holds:
$$
\varphi_{\alpha}(a) 
 := \int_0^{\infty}  \alpha(\tau(e_{(s,\infty)}(a)))ds = \int_0^1 \lambda_t(a) d\nu_{\alpha}(t). 
$$
\end{proposition}
\begin{proof}
Let $a \in \M^+$. If $a = 0$, then it is clear.  So we assume that $a \not= 0$. 
First, consider the case that  $a \in \M^+$ is of $\tau$-finite rank in $\M$.  Let 
$a = \sum_{k=1}^{n} a_k p_k$ be  the spectral decomposition of $a$ with $a_1 > a_2> \dots > a_n >0$ and each $\tau(p_k) < \infty$ for 
$k = 1,2,\dots,n$. Then 
\begin{align*}
     & \alpha(\tau(e_{(s,\infty)}(a))) \\
  = &
\begin{cases}
    0,& (a_1 \leq s ), \\
    \alpha(\tau(p_1 + p_2 + \dots + p_k)), & (a_{k+1} \leq s < a_k, 1\le k\le n-1), \\
    \alpha(\tau(p_1 + p_2 + \dots + p_n)), & (0 \leq s < a_n).
\end{cases}
\end{align*}
Therefore 
\begin{align*}
    & \varphi_{\alpha}(a) 
  = \int_0^{\infty}  \alpha(\tau(e_{(s,\infty)}(a)))ds \\
  = &   \sum_{k=1}^{n-1} (a_k- a_{k+1})\alpha(\tau(p_1 + \cdots + p_k))
    + a_n \alpha(\tau(p_1 + \cdots + p_n)).
\end{align*}
On the other hand, 
\begin{align*}
    & \lambda_t(a) = \inf \{s \geq 0 | \tau(e_{(s, \infty)}(a)) \leq t \} \\
  = &  \begin{cases}
           a_1,& (0 \leq t < \tau(p_1), \\
           a_k & (\tau(p_1 + \cdots + p_{k-1}) \leq  t < \tau(p_1 + \dots + p_k),  2\le k \le n), \\
           0    & (\tau(p_1 + \cdots +p_n) \leq t).
        \end{cases}
\end{align*}
Therefore 
\begin{align*}
     & \int_0^{\infty} \lambda_t(a) d\nu_{\alpha}(t) \\
  =  & a_1(\alpha(\tau(p_1)) -\alpha(0)) + 
        \sum_{k=2}^{n} a_k(\alpha(\tau(p_1 + \cdots + p_k)) - \alpha(\tau(p_1 + \cdots + p_{k-1}))) \\
  =  & \sum_{k=1}^{n-1} (a_k- a_{k+1})(\alpha(\tau(p_1 +  \cdots + p_k)))
            + a_n (\alpha(\tau(p_1 +  \cdots  + p_n)))
        = \varphi_{\alpha}(a) .
\end{align*}

Secondly, consider the case that  $a \in \M^+$ has a finite spectrum.  Let 
$a = \sum_{k=1}^{n} a_k p_k$ be  the spectral decomposition of $a$ with $a_1 > a_2> \dots > a_n >0$ for 
$k = 1,2,\dots,n$ . Then 
\[  \int_0^{\infty}  \alpha(\tau(e_{(s,\infty)}(b_k)))ds = \infty  \]
if and only if 
\[   \alpha(\tau(p_i)) = \infty  \text{ for some }i   = 1,2,\dots,n  \]
if and only if 
\[  \tau(p_i) = \infty \text{ and }   \lim_{t \to \infty} \alpha(t) = \infty  \text{ for some }i   = 1,2,\dots,n  \]
if and only if 
\[  \int_0^{\infty}  \lambda_t(b_k) d\nu_{\alpha}(t) = \infty.   \]
Next, consider the general case that $a \in \M^+$. Then there exists an increasing sequence 
$(b_n)_n$ in $\M^+$ of a finite spectrum such that $\| b_n -a\| \rightarrow 0$. By \cite[page 277]{F-K}, 
for any $x,y \in \M$
$$
|\mu_t(x) - \mu_t(y)| \leq \| x - y \|.
$$
Therefore $\lambda_t(b_n)$ converges to  $\lambda_t(a)$ pointwise in $t$. By the monotone convergence theorem, 
we have that 
$$
\int_0^{\infty} \lambda_t(b_n) d\nu_{\alpha}(t)  \nearrow   \int_0^{\infty} \lambda_t(a) d\nu_{\alpha}(t) .
$$
Since  $\varphi_{\alpha}$ is sequentially normal. Thus $\varphi_{\alpha}(b_n)  \nearrow \varphi_{\alpha}(a)$. 
If all $b_n$ are of $\tau$-finite rank, then the equation holds. If some $b_k$ is not of $\tau$-finite rank, 
Then 
$$
\int_0^{\infty}  \alpha(\tau(e_{(s,\infty)}(b_k)))ds = \infty = \int_0^{\infty} \lambda_t(b_k) d\nu_{\alpha}(t). 
$$
This implies the conclusion in the general case. The rest is clear. 
\end{proof}

%\begin{remark} \rm
%Let $a \in M^+$.  In general, 
%$$
% \int_0^{\infty}  \alpha(\tau(e_{[s,\infty)}(a)))ds  
% \not=  \int_0^{\infty}  \alpha(\tau(e_{(s,\infty)}(a)))ds 
%$$
%But if $M$ is a finite factor, then we claim that 
%\begin{align*}
%\varphi_{\alpha}(a) 
% & = \int_0^{\infty}  \alpha(\tau(supp((a-sI)_+)))ds \\
% & = \int_0^{\infty}  \alpha(\tau(e_{[s,\infty)}(a)))ds  
% =  \int_0^{\infty}  \alpha(\tau(e_{(s,\infty)}(a)))ds 
%\end{align*}
%In fact, for $s > 0$ and $\varepsilon > 0$, 
%\begin{align*}
%\int_0^{\infty}  \alpha(\tau(e_{[s,\infty)}(a)))ds 
%& \geq \int_0^{\infty}  \alpha(\tau(e_{(s,\infty)}(a)))ds \\
%& \geq \int_0^{\infty}  \alpha(\tau(e_{[s + \varepsilon,\infty)}(a)))ds 
%= \int_{\varepsilon}^{\infty}  \alpha(\tau(e_{[s,\infty)}(a)))ds  \\
%&  =  \int_0^{\infty}  \alpha(\tau(e_{[s,\infty)}(a)))ds  - \int_0^{\varepsilon}  \alpha(\tau(e_{[s,\infty)}(a)))ds \\
%& \geq  \int_0^{\infty}  \alpha(\tau(e_{[s,\infty)}(a)))ds  - \varepsilon \alpha(\tau(I)).
%\end{align*}
%Since $\alpha(\tau(I)) < \infty$, letting $\varepsilon \to 0$, we get the claim.  
%\end{remark}

%%%%%%%%%%%%%%%%%%%%%%%%
%   Example 2.25
%%%%%%%%%%%%%%%%%%%%%%%%
\begin{example} \rm
Let $\alpha(x) = x$, then we can recover the trace $\tau$,  and for any $a \in \M^+$, 
$$
\varphi_{\alpha}(a) 
 = \int_0^{\infty}  \tau(e_{(s,\infty)}(a))ds = \int_0^{\infty} \lambda_t(a) dt =\tau(a). 
$$
\end{example}

\begin{example} \rm
Let  $$
\alpha(x) 
= \begin{cases}
1    & (0 < x), \\
0    & (x = 0), 
\end{cases}
$$
then we can recover the operator norm and for any $a \in \M^+$,  and we have that 
$$
\varphi_{\alpha}(a) 
 = \int_0^{\infty}  \alpha(\tau(e_{(s,\infty)}(a)))ds = \int_0^{\infty} \lambda_t(a) d\nu_{\alpha}(t) = ||a||.
$$
\end{example}

\begin{example} \rm
Fix $t > 0$ and 
let  $$
\alpha(x) 
= \begin{cases}
x    & (0 \leq x \leq t ), \\
t    & (t < x ), 
\end{cases}
$$
then we obtain a continuous analog of  the Ky Fan norm and for any $a \in \M^+$,  we have that 
$$
\varphi_{\alpha}(a) 
 = \int_0^{\infty}  \alpha(\tau(e_{(s,\infty)}(a)))ds = \int_0^{t} \lambda_t(a) dt.   
$$
\end{example}

%%%%%%%%%%%%%%%%%%%%%%%%%%%
%   Example 2.28
%%%%%%%%%%%%%%%%%%%%%%%%%%%
\begin{example} \rm
Let  $$
\alpha(x) 
= \begin{cases}
0    & (x = 0), \\
1    & (0 < x \le 2), \\
4    & (2 < x ), 

\end{cases}
$$
then for any $a \in \M^+$,  and we have that 
$$
\varphi_{\alpha}(a) 
=\lambda_0(a) + 3\lambda_2(a) =  \|a\| +  3\lambda_2(a).
$$
Then $\varphi_{\alpha}(|a|)$ does not give a norm but gives a quasi-norm. In fact, 
take orthogonal projections $p$ and $q$ such that $\tau(p) = \tau(q) =4/3 $. Then 
$\varphi_{\alpha}(p) = \varphi_{\alpha}(q) = 1 + 3 \times 0 = 1$ and 
$\varphi_{\alpha}(p + q) = 1 + 3 \times 1 = 4$. Hence 
$\varphi_{\alpha}$ does not give a norm .  
Since for positive $a,b \in \M$, 
$\lambda_0(a + b) \leq \lambda_0(a) + \lambda_0(b)$, 
$\lambda_2(a + b) \leq \lambda_1(a) + \lambda_1(b)$, 
$\lambda_2(a + b) \leq \lambda_0(a) + \lambda_2(b)$,  and 
$\lambda_2(a + b) \leq \lambda_2(a) + \lambda_0(b)$ 
 by Weyl's inequality (the property (3) of $\mu_t$), 
we have that 
\begin{align*}
 \varphi_{\alpha}(a + b) 
& =\lambda_0(a +b) + 3\lambda_2(a + b) \\
& \leq  \lambda_0(a) + \lambda_0(b) + \frac{3}{2}( \lambda_0(a) + \lambda_2(b) + \lambda_2(a) + \lambda_0(b))\\
&\leq  \frac{5}{2} (\varphi_{\alpha}(a) + \varphi_{\alpha}(b) ).
 \end{align*} 
Hence $\varphi_{\alpha}(|a|)$ gives a quasi-norm.

\end{example}

We characterize non-linear traces of the Choquet type on semifinite factors. 

%%%%%%%%%%%%%%%%%%%%%%%%%%%%%
%     Theorem 2.29
%%%%%%%%%%%%%%%%%%%%%%%%%%%%%
\begin{theorem} 
Let $\M$ be a factor of type ${\rm II}_{\infty}$ with a faithful normal semifinite trace $\tau$. 
Let $\varphi : \M^+ \rightarrow  [0, \infty]$ be a non-linear 
positive map.  Then the following conditions are equivalent:
\begin{enumerate}
\item[$(1)$]  $\varphi$ is a non-linear trace $\varphi = \varphi_{\alpha}$  of the Choquet type associated with  
a monotone increasing and left continuous function $\alpha: [0, \infty) \rightarrow [0, \infty)$  with $\alpha(0) = 0$.
\item[$(2)$] $\varphi $ is monotonic increasing additive on the spectrum, unitarily invariant, monotone, 
positively homogeneous,  and sequentially normal on $\M^+$.  Moreover,  $\varphi(a) < \infty  $ 
for any $a \in \M^+_{\tau-fr}$ . 
\end{enumerate}
Let $\M$ be a factor of type ${\rm II}_1$ with a faithful normal finite trace with $\tau(1) = 1$. 
Let $\varphi : \M^+ \rightarrow  [0, \infty]$ be a non-linear 
positive map.  Then similarly, the following conditions are equivalent:
\begin{enumerate}
\item[$(1)$]  $\varphi$ is a non-linear trace $\varphi = \varphi_{\alpha}$  of the Choquet type associated with  
a monotone increasing and left continuous function $\alpha: [0, 1] \rightarrow [0, \infty)$  with $\alpha(0) = 0$.
\item[$(2)$] $\varphi $ is monotonic increasing additive on the spectrum, unitarily invariant, monotone, and 
positively homogeneous.  Moreover,  $\varphi(a) < \infty  $ 
for any $a \in \M^+_{\tau-fr}$. 
(We do not need that  $\varphi $ is sequentially normal on $\M^+$.) 
\end{enumerate}
\end{theorem}
\begin{proof}
First, consider the cast that $\M$ is a factor of type ${\rm II}_{\infty}$.
  
$(1) \Rightarrow (2)$ Assume that $\varphi$ is a non-linear trace $\varphi = \varphi_{\alpha}$  of the Choquet type associated with  
a monotone increasing function $\alpha$. Then we have already proved that $\varphi_{\alpha}$ is unitarily invariant, monotone,  
and sequentially normal on $\M^+$.  
It is obvious that $\varphi_{\alpha}$ is positively homogeneous by the definition using the change of variable. 
Let $a \in \M^+$ be of $\tau$-finite rank in $\M$ and 
$a = \sum_{k=1}^{n} a_k p_k$ be  the spectral decomposition of $a$ with $a_1 > a_2 >  \dots > a_n >0$ and each $\tau(p_k) < \infty$ for 
$k = 1,2,\dots,n$.
Then 
\begin{align*}
     &  \alpha(\tau(e_{(s,\infty)}(a)))\\
 =  &  \begin{cases}
           \alpha(\tau(0)), & (s \geq a_1) \\
           \alpha(\tau(p_1)), & (a_1 > s \geq a_2) \\
           \alpha(\tau(p_1 +  \dots + p_k)), & (a_k > s \geq a_{k+1}, \   2\le k \le n-1) \\
           \alpha(\tau(p_1 +  \dots + p_n)), & (a_n > s \geq 0) \\
        \end{cases} .
\end{align*}
Therefore we have that 
$$
\varphi_{\alpha}(a) 
 = \sum_{k=1}^{n-1} (a_k- a_{k+1})\alpha(\tau(p_1 + \cdots + p_k))
+ a_n \alpha(\tau(p_1 + \cdots + p_n))  < \infty.
$$
We should note that this formula also holds in the case that  $a_1 \geq  a_2 \geq  \dots \geq a_n  \geq 0$, 
because some terms just become 0.   
We show that $\varphi_{\alpha} $ is monotonic increasing additive on the spectrum.  

First,  assume that $a\in \M^+$ has a finite spectrum and $a = \sum_{k=1}^{n} a_k p_k$ is the spectral decomposition of $a$ with $a_1 \geq a_2 \geq \dots \geq  a_n >0$.  
For any monotone increasing functions $f$ and $g$  in $C(\sigma(a))^+$, 
\begin{gather*}
   f(a) = \sum_{k=1}^{n} f(a_k) p_k, \ g(a) = \sum_{k=1}^{n} g(a_k) p_k, \\
   (f + g)(a) = \sum_{k=1}^{n} (f(a_k) + g(a_k)) p_k,  \\ 
   f(a_1) \geq f(a_2) \geq  \cdots \geq f(a_n) \geq 0, \\ g(a_1) \geq g(a_2) \geq \cdots \geq  g(a_n) \geq 0, \\
  \text{and }  f(a_1)+g(a_1)  \geq f(a_2) + g(a_2) \geq  \cdots \geq  f(a_n) + g(a_n) \geq 0. 
\end{gather*}
%and  $f(a_1)+g(a_1)  \geq f(a_2) + g(a_2) \geq  \cdots \geq  f(a_n) + g(a_n) \geq 0$. 
Therefore 
\begin{align*}
     & \varphi_{\alpha}((f+g)(a)) \\
   = & \sum_{k=1}^{n-1} (f(a_k) + g(a_k)- f(a_{k+1})-g(a_{k+1}))\alpha(\tau(p_1 + \cdots + p_k)) \\
     & \qquad +( f(a_n) + g(a_n)) \alpha(\tau(p_1 + \cdots + p_n)) \\
   = & \sum_{k=1}^{n-1} (f(a_k) - f(a_{k+1}))\alpha(\tau(p_1 + \cdots + p_k))
             + f(a_n) \alpha(\tau(p_1 + \cdots + p_n)) \\
     &  + \sum_{k=1}^{n-1} (g(a_k) - g(a_{k+1}))\alpha(\tau(p_1 + \cdots + p_k))
         + g(a_n) \alpha(\tau(p_1 + \cdots + p_n)) \\
  = & \varphi_{\alpha}(f(a)) +  \varphi_{\alpha}(g(a)).
\end{align*}
Next, consider general $a\in \M^+$.  
Then there exists an increasing sequence $(b_n)_n$ in $\M^+$ such that each 
$b_n$ has a finite spectrum and $\| b_n - a\|$ converges to $0$.  
Moreover,  there exists a compact subset 
$K \subset {\mathbb R}$ such that for any $n$, $\sigma(b_n) \subset K$ and $\sigma(a) \subset K$. 
For any monotone increasing functions $f$ and $g$  in 
$C([0,\infty))^+$, we have that 
$$
\varphi_{\alpha}((f+g)(b_n)) =  \varphi_{\alpha}(f(b_n)) +  \varphi_{\alpha}(g(b_n)).
$$
We note that $f(b_n) \nearrow f(a)$, $g(b_n) \nearrow g(a)$ and  $(f+g)(b_n) \nearrow (f+g)(a)$. 
Since $\varphi_{\alpha}$ is sequentially normal, 
$$
\varphi_{\alpha}((f+g)(a)) =  \varphi_{\alpha}(f(a)) +  \varphi_{\alpha}(g(a)).
$$
Thus $\varphi_{\alpha} $ is monotonic increasing additive on the spectrum.

%\noindent
$(2)\Rightarrow(1)$
 Assume that $\varphi $ is monotonic increasing additive on the spectrum, unitarily invariant, monotone,  and 
positively homogeneous. 
We also assume that $\varphi $ is sequentially normal and 
 $\varphi(a) < \infty  $ for any $\tau$-finite rank operator $a \in \M^+_{\tau-fr}$.  
Define  a function $\alpha: [0,\infty) \rightarrow [0, \infty)$  with $\alpha(0) = 0$
by $\alpha (x) := \varphi(p_x)$ for some projection $p_x$ in $\M$ with $\tau(p_x) = x \ (x \in [0,\infty)).$  
Since such a projection $p_x$ is of $\tau$-finite rank, $\alpha (x) = \varphi(p_x) < \infty$.  
Since  $\varphi $ is unitarily invariant, $\alpha$ does not depend 
on the choice of such projections. Since  $\varphi $ is monotone, $\alpha$ is monotone increasing. 
Since $\varphi $ is sequentially normal,  $\alpha$ is left continuous. 

We  show that $\varphi = \varphi_{\alpha}$.  We only consider the case that $\M$ is an infinite factor. 
If $\M$ is a finite factor, its proof is a little bit simpler.  
First assume that $a\in \M^+$ has a finite spectrum $\sigma(a) = \{a_1, a_2, \dots ,a_n \}$ 
and $a = \sum_{k=1}^{n} a_k p_k$ be  the spectral decomposition of $a$ with 
$a_1 \geq a_2 \geq \dots \geq  a_n >0$. 
Suppose that $\tau(p_k) = \infty$ for some $k$. 
Since $0 \leq \frac{1}{a_k}p_k \leq a$, monotonicity and positive homogeneity imply that 
$\varphi (a) = \infty  = \varphi_{\alpha}(a)$. 
Therefore we may and do assume that  
$\tau(p_k) < \infty$ for any $k$, that is, $a$ is of $\tau$-finite rank.  
Now define functions  $f_1,f_2,\dots, f_n \in C(\sigma(a))^+$  by 
$$
f_i (x) = \begin{cases}    a_i- a_{i+1}, \quad & \text{ if } x\in  \{ a_1, a_2, \dots, a_i \}  \\
                                 0  &  \text{ if }  x \in \{ a_{i+1}, \dots, a_n, 0 \}
\end{cases}
$$
for $i = 1,2,\dots, n-1$ and
$$
f_n (x) = \begin{cases}    a_n  & \text{ if } x \in  \{a_1, a_2, \dots, a_n \}  \\
0  &  \text{ if }  x \in \{ 0 \}. 
\end{cases}
$$
Then  $f_1, f_2, \dots, f_n$ are monotone increasing functions in  $C(\sigma(a))^+$  for $i=1,2,\ldots,n$ 
such that 
$$
f_1(x) + f_2(x) + \dots +f_n(x) = x  \ \ \ \text{ for } x \in \sigma(a). \\
$$
Therefore 
$$
f_1(a) + f_2(a) + \dots +f_n(a) = a .
$$
Moreover,  we have that 
$$
f_i(a) =  (a_i- a_{i+1})(p_1 + p_2 + \dots +p_i) \quad \text{ for } i=1,2,\dots, n-1
$$
and $f_n(a) =  a_n (p_1 + p_2 + \dots + p_n).$

Since $\varphi $ is monotonic increasing additive on the spectrum and 
positively homogeneous, we have that 
\begin{align*}
   & \varphi(a)  = \varphi(f_1(a) + f_2(a) + \cdots +f_n(a)) \\
  =  &  \varphi(f_1(a)) + \varphi(f_2(a)) + \cdots +\varphi(f_n(a)) \\
  =  &  \sum_{i=1} ^{n-1} \varphi( (a_i- a_{i+1})(p_1 + p_2 + \cdots +p_i)) 
       + \varphi(a_n (p_1 + p_2 + \cdots +p_n))\\
  =  &  \sum_{i=1}^{n-1}  ( a_i- a_{i+1})\varphi(p_1 + p_2 + \cdots + p_i) 
       + a_n \varphi (p_1 + p_2 + \cdots +p_n)   \\
  =  &  \sum_{i=1} ^{n-1}  ( a_i- a_{i+1})(\alpha(\tau(p_1 + p_2 + \dots + p_i)))
       + a_n \alpha(\tau (p_1 + p_2 + \cdots +p_n)) \\
  =  & \varphi_{\alpha}(a).
\end{align*}
Next,  we consider general $a\in \M^+$.  
Then there exists an increasing sequence $(b_n)_n$ in $\M^+$ such that each 
$b_n$ has a finite spectrum and $\| b_n - a\|$ converges to $0$. 
We may assume that each $b_n$ is of 
$\tau$-finite rank.  
Otherwise, $\varphi (a) = \infty  = \varphi_{\alpha}(a)$.  
Both $\varphi$ and $\varphi_{\alpha}$ are sequentially normal and 
$\varphi (b_n) = \varphi_{\alpha}(b_n)$.  Therefore $\varphi (a) = \varphi_{\alpha}(a)$.

The case of a factor of type ${\rm II}_1$ is similarly proved.  But 
we do not need the condition that $\varphi_{\alpha}$ is sequentially normal on $\M^+$ of (2).   
In fact, take $a\in \M^+$ with the spectral decomposition $a =  \int_0^{\infty} \lambda d e_{\lambda}(a)$. 
Define $b_n\in \M^+$ and $c_n\in \M^+$ with 
$$
b_n = \sum_{k=1}^{n-1} \frac{k-1}{n}\|a\| e_{[\frac{k-1}{n}\|a\|, \frac{k}{n}\|a\|)}
         + \frac{n-1}{n}\|a\| e_{[\frac{n-1}{n}\|a\|, \frac{n}{n}\|a\|]}, 
$$
$$
c_n = \sum_{k=1}^{n-1} \frac{k}{n}\|a\|e_{[\frac{k-1}{n}\|a\|, \frac{k}{n}\|a\|)}
         + \frac{n}{n}\|a\|e_{[\frac{n-1}{n}\|a\|, \frac{n}{n}\|a\|]}.  
$$
Then $0 \leq b_n \leq a \leq c_n$ and $0 \leq \varphi(b_n) \leq \varphi(a) \leq \varphi(c_n) < \infty$ by monotonicity. 
Since 
$$
\varphi(c_n) - \varphi(b_n) = \frac{1}{n}\|a\|\alpha(\tau(I)) \rightarrow 0
$$
We have that  $\varphi(b_n) \rightarrow \varphi(a)$.  Similarly  $\varphi_{\alpha}(b_n) \rightarrow \varphi_{\alpha}(a) $. 
Since $\varphi (b_n) = \varphi_{\alpha}(b_n)$, this implies that $\varphi (a) = \varphi_{\alpha}(a)$.
\end{proof}

We shall consider a similar statement on quasi-norms like as Theorem \ref{thm:qnorm} for a factor $\M$ of type ${\rm II}_1$. 
We assume that $\alpha$ is continuous for a while.

%
%  Lemma 2.30
%
\begin{lemma}
\label{lem:series}
Let $\M$ be a factor of type ${\rm II}_1$ and $a \in \M^+$. Assume that $\alpha$ is continuous.  
Then, for any $\varepsilon > 0$, there exists a positive integer $M$ such that
\[   \varphi_\alpha (a) \le \sum_{i=1}^M \lambda_{(i-1)/M}(a) ( \alpha( i/M ) - \alpha( (i-1)/M ) ) < \varphi_\alpha(a) + \epsilon. \]
\end{lemma}
\begin{proof}
Since $\varphi_\alpha(a) = \int_0^1 \lambda_s(a) d\nu_\alpha(s)$, we can choose a partition of $[0,\|a\|]$ as follows:
\[    \|a\|=a_0 >a_1>a_2>\cdots > a_N=0  \]
and
\begin{align*}
    \varphi_\alpha(a) & = \int_0^1 \lambda_s(a) d\nu_\alpha(s)  \\
        & \le \sum_{i=1}^N a_{i-1} \nu_\alpha (\{t : a_i<\lambda_t(a) \le a_{i-1} \}) < \varphi_\alpha(a) + \epsilon .
\end{align*}
Since the function $\lambda_{-}(a) :[0,1] \ni t \mapsto \lambda_t(a)$ is decreasing and right-continuous,  there exists a partition 
$\Delta : 0=t_0<t_1 \le t_2 \le \cdots \le t_N \le t_{N+1} =1$ of $[0,1]$ such that
\[    \{ t: a_i<\lambda_t(a) \le a_{i-1} \} = [t_{i-1}, t_i )    \]
and $\nu_\alpha ([t_{i-1}, t_i) ) = \alpha(t_i) - \alpha(t_{i-1})$,  $(i=1,2,\ldots, N+1)$ .
So we have
\[  \varphi_\alpha(a) \le \sum_{i=1}^{N+1} a_{i-1} (\alpha(t_i) - \alpha(t_{i-1}) )  < \varphi_\alpha(a) + \epsilon. \]

Put $t'_i$ $(i=0,1,2,\ldots, N+1)$ as follows:
\[   t_i' = \begin{cases} t_i  & t_i \in {\rm Im}\lambda_{-}(a)  \\
               \inf \{ t :  \lambda_t(a) \le a_i \}& t_i \notin {\rm Im}\lambda_{-}(a) \end{cases}.  \]
Then  $0=t_0' <t_1' \le t_2' \le \cdots \le t_N' \le t_{N+1}'=1$,
\[  \lambda_{s}(a) \le \sum_{i=1}^{N+1} \lambda_{t_{i_1}'}(a) \chi_{[t_{i-1}', t_i')}(s) \le \sum_{i=1}^N a_{i-1} \chi_{[t_{i-1}', t_i')}(s) , \]
and
\[  \int_0^1 \lambda_s(a) d\nu_\alpha(s) \le \sum_{i=1}^{N+1} \lambda_{t_{i-1}'} (\alpha(t_i') - \alpha(t_{i-1}') ) \le \sum_{i=1}^{N+1} a_{i-1} (\alpha(t_i) - \alpha(t_{i-1}) ). \]
We remark that, for any refinement 
$\tilde{\Delta} : \tilde{t_0} \le \tilde{t_1} \le \cdots \le \tilde{  t_{L} } $ of the partition $\Delta' : t_0' <t_1' \le t_2' \le \cdots \le t_N' \le t_{N+1}'$
\[  \int_0^1 \lambda_s(a) d\nu_\alpha(s) \le \sum_{i=1}^{L} \lambda_{\tilde{t_{i-1}}} (\alpha(\tilde{t_i}) - \alpha(\tilde{t_{i-1}}) ) 
            \le \sum_{i=1}^{N+1} \lambda_{t_{i-1}'} (\alpha(t_i') - \alpha(t_{i-1}') ),
\] 
because $\lambda_{-}(s)$ is decreasing.

Using the right-continuity of $\lambda_{-}(a)$ and $\alpha$, there exists $\delta>0$ and a partition $\Delta': t_0''\le t_1'' \le \cdots \le t_{N+1}''$ 
such that $t_i''-t_i'<\delta$ and
\[  \varphi_\alpha(a) =  \int_0^1 \lambda_s(a) d\nu_\alpha(s)  \le \sum_{i=1}^{N+1} \lambda_{t_{i_1}''}(a)  (\alpha(\tilde{t_i}'') - \alpha(\tilde{t_{i-1}}'') )
       < \varphi_\alpha(a) + \epsilon. \]
So we can choose $\{t_0'', t_1'', \ldots , t_{N+1}'' \} \subset \{ i/M \;| \; i=0,1,2,\ldots, M\}$ for some positive integer $M$.
This means that
\[   \varphi_\alpha(a)  \le \sum_{i=1}^{M} \lambda_{(i-1)/M}(a)  (\alpha(i/M) - \alpha((i-1)/M) )
       < \varphi_\alpha(a) + \epsilon. \]
\end{proof}

%
%  Theorem 2.31
%
\begin{theorem}
Let $\M$ be a factor of type ${\rm II}_1$,  $\alpha :[0,1]\rightarrow [0,\infty)$ be a monotone increasing continuous function with $\alpha(0)=0$, and
$\varphi_\alpha : \M^+\rightarrow [0,\infty]$ be a non-linear trace of the Choquet type  associated with $\alpha$. 
Then the following are equivalent:
\begin{enumerate}
  \item[$(1)$] There exists a positive number $\beta$ such that
\[    \varphi_\alpha(|a+b|) \le \beta ( \varphi_\alpha(|a|) + \varphi_\alpha (|b|) ) \text{ for any }a,b \in \M.  \]
  \item[$(2)$] $\sup \{ \dfrac{ \alpha(2s) }{\alpha(s)} \; | \;s\in[0,1] \}$ is bounded, where we set $\alpha(s)=\alpha(1)$ if $s>1$.
  \item[$(3)$] $\M\ni a \mapsto |||a|||_{\alpha,p}(a) = \varphi_\alpha(|a|^p)^{1/p} $ is a quasi-norm for any $p>0$.
  \item[$(4)$] $\M\ni a \mapsto |||a|||_{\alpha,p}(a) = \varphi_\alpha(|a|^p)^{1/p} $ is a quasi-norm for some $p>0$. 
\end{enumerate}
\end{theorem}
\begin{proof}
$(1)\Rightarrow(2)$ 
It suffices to show that $\{ \frac{\alpha(2s)}{\alpha(s)} \;| \; s \in[0, 1/2] \}$ is bounded.
Since $\M$ is a factor of type ${\rm II}_1$, for $s\in[0, 1/2]$, there exist orthogonal projections $p, q\in \M$ such that
\[    \tau(p)=\tau(q)=s  \text{ and } \tau(p+q)=2s  .  \]
Then the relation
\[  \alpha(2s) = \varphi_\alpha(p+q) \le \beta (\varphi_\alpha(p)+ \varphi_\alpha(q)) =2\beta \alpha(s) \]
implies $\frac{\alpha(2s)}{\alpha(s)} \le 2\beta$ for all $s \in[0, 1/2] $.

$(2)\Rightarrow(1)$ We may assume that $a,b\in \M^+$ and put $\beta = \sup\{ \frac{\alpha(2s)}{\alpha(s)} \ |  s \in [0,1] \}$.
By Proposition \ref{prop:Stieltjes}, 
\begin{gather*}
 \varphi_\alpha(a)= \int_0^1 \lambda_t(a) d\nu(t),  \; \varphi_\alpha(b)= \int_0^1 \lambda_t(b)d\nu(t), \\
 \varphi_\alpha(a+b)=\int_0^1 \lambda_t(a+b)d\nu(t)  
\end{gather*} 
for $a, b \in \M^+$.
For any $\epsilon >0$, we can choose a positive integer $L$ such that
 \begin{gather*}
    0 \le  \sum_{i=1}^L \lambda_{(i-1)/L}(a) (\alpha(i/L)-\alpha((i-1)/L) - \varphi_\alpha(a)<\epsilon  \\
    0 \le  \sum_{i=1}^L \lambda_{(i-1)/L}(b) (\alpha(i/L)-\alpha((i-1)/L) - \varphi_\alpha(b)<\epsilon  \\
    0 \le  \sum_{i=1}^L \lambda_{(i-1)/L}(a+b)(\alpha(i/L)-\alpha((i-1)/L) - \varphi_\alpha(a+b)<\epsilon  
\end{gather*}
by Lemma \ref{lem:series}.
Using Weyl's inequality (the property (3) of $\mu_t$), we have
\[   \lambda_{2i/L}(a+b) \le \lambda_{i/L}(a)+\lambda_{i/L}(b), \lambda_{(2i+1)/L}(a+b) \le \lambda_{i/L}(a) + \lambda_{i/L}(b)  \]
for $i=0,1,2,\ldots, L$.
We define $c(t)$, $s_t(a)$, $s_t(b)$, and $s_t(a+b)$ for $t\ge 1$ as follows:
\[  c(t)=\alpha(t/L) -\alpha((t-1)/L), \;  s_t(x)=\lambda_{(t-1)/L}(x),   \]
where $x=a,b,$ or $a+b$, then
\begin{gather*}
  s_i(x)=0 \; (i\ge L+1), \quad   \sum_{i=1}^n c(i) = \alpha(n/L) \; (n\in \mathbb{N}),  \\
  \text{and } \max\{ s_{2i-1}(a+b), s_{2i}(a+b)\}  \le s_i(a)+s_i(b)  .
\end{gather*}
By the similar argument in the Theorem \ref{thm:qnorm},
\begin{align*}
        & \varphi_\alpha (a+b)  \\ 
   \le & \sum_{i=1}^L \lambda_{(i-1)/L} (a+b) (\alpha(i/L)-\alpha((i-1)/L)  +\epsilon  \\
   =   &  \sum_{i=1}^L s_{i}(a+b) c(i) + \epsilon =  \sum_{i=1}^{2L} s_{i}(a+b) c(i) + \epsilon \\
   =   &  \sum_{i=1}^L s_{2i-1}(a+b)c(2i-1) + \sum_{i=1}^L s_{2i}(a+b) c(2i) +\epsilon \\
   \le &  \sum_{i=1}^L (s_i(a) + s_i(b))(c(2i-1)+c(2i))+\epsilon \\
   =   &  \sum_{i=1}^L \bigl(  (\sum_{j=1}^{2i} c(j) ) (s_i(a)-s_{i+1}(a) ) +  (\sum_{j=1}^{2i} c(j) ) (s_i(b)-s_{i+1}(b) ) \bigr) +\epsilon  \\
   =   &  \sum_{i=1}^L \bigl(  \alpha(2i/L) (s_i(a)-s_{i+1}(a) ) +  \alpha(2i/L) (s_i(b)-s_{i+1}(b) ) \bigr) +\epsilon  \\
   \le & \beta \sum_{i=1}^L \bigl(  \alpha(i/L) (s_i(a)-s_{i+1}(a) ) +  \alpha(i/L) (s_i(b)-s_{i+1}(b) ) \bigr) +\epsilon  \\
   =   & \beta \sum_{i=1}^L \bigl(  (\sum_{j=1}^{i} c(j) ) (s_i(a)-s_{i+1}(a) ) +  (\sum_{j=1}^{i} c(j) ) (s_i(b)-s_{i+1}(b) ) \bigr) +\epsilon  \\
   =   & \beta (\sum_{i=1}^L s_i(a)c(i) + \sum_{i=1}^L s_i(b)c(i) ) +\epsilon \\
   \le & \beta (\varphi_\alpha(a) + \varphi_\alpha(b) +2\epsilon) + \epsilon .
\end{align*}
Since $\epsilon$ is arbitrary, we have
\[   \varphi_\alpha(a+b) \le \beta (\varphi_\alpha(a) + \varphi_\alpha(b) ).  \]

Using Lemma \ref{lem:series}, we can prove $(1)\Leftrightarrow(2)$ by the similar method to Theorem \ref{thm:qnorm} as above.
For any $a\in \M$, $p>0$, and $\epsilon >0$, we can choose a positive integer $M$ such that
\[  \varphi(|a|^p) \le \sum_{i=1}^M \mu_{(i-1)/M}(a)^p ( \alpha( i/M ) - \alpha( (i-1)/M ) ) < \varphi_\alpha(|a|^p) + \epsilon. \]
Apply the argument of Theorem \ref{thm:qnorm} for 
\[    \sum_{i=1}^M \mu_{(i-1)/M}(a)^p ( \alpha( i/M ) - \alpha( (i-1)/M ) ),  \]
we can also prove $(1)\Leftrightarrow(3)\Leftrightarrow(4)$.
\end{proof}

%
%  Remark 2.32
%
\begin{remark} \rm
Let $\M$ be a factor of type ${\rm II}_\infty$. Fix $ p > 0$.  For $a\in \M^+$ with $\tau({\rm supp}(a))<\infty$.
By a similar argument in Lemma \ref{lem:series}, for any $\epsilon>0$, we can choose positive integers $L$ and $M$ such that
$\tau({\rm supp}(a)) \le M/L$ and
\[  ( \varphi(|a|^p) )^{1/p} \le ( \sum_{i=1}^M \mu_{(i-1)/L}(a)^p ( \alpha( i/L) - \alpha( (i-1)/L ) ) )^{1/p} < ( \varphi_\alpha(|a|^p) )^{1/p} + \epsilon. \]

So we can show that the following are equivalent:
\begin{enumerate}
  \item[$(1)$]  There exists a positive number $\beta$ such that
\[    (\varphi_\alpha(|a+b|^p))^{1/p} \le \beta ( (\varphi_\alpha(|a|^p))^{1/p} + (\varphi_\alpha (|b|^p))^{1/p} )   \]
for any $a,b \in \M$ with $\tau({\rm supp}(a)),\tau({\rm supp}(b))<\infty$. 
  \item[$(2)$]   $\sup \{ \dfrac{ \alpha(2s) }{\alpha(s)} \; | \;s\in[0,\infty) \}$ is bounded,.
\end{enumerate}

Moreover, for a fixed $p$ with $p \geq 1$, applying Theorm \ref{thm:norm} to the estimation of $\varphi_\alpha(|a|^p)^{1/p}$,
we can also show that the following conditions are equivalent:
%for a fixed $p$ with $p \geq 1$, for any  $a\in \M^+$  with $\tau({\rm supp}(a))<\infty$ and any $\epsilon >0$, we can also %choose positive integers $L$ and $M$ such that $\tau({\rm supp}(a)) \le M/L$ and
%\[  ( \varphi(|a|^p) )^{1/p} \le ( \sum_{i=1}^M \mu_{(i-1)/M}(a)^p ( \alpha( i/M ) - \alpha( (i-1)/M ) ) )^{1/p} < ( %\varphi_\alpha(a) )^{1/p} + \epsilon. \]

\begin{enumerate}
  \item[(a)]  $\alpha$ is concave.
  \item[(b)]  $|||a+b|||_{\alpha, p} \le |||a|||_{\alpha, p} + |||b|||_{\alpha, p}$, 
\end{enumerate}
for any $a, b\in \M$  with $\tau({\rm supp}(a)),\tau({\rm supp}(b))<\infty$.  
\end{remark}

\begin{example}\rm 
Let $\M$ be a factor of type ${\rm II}_\infty$ and fix $p > 0$. Put $\alpha(s) = s^2$ for $s \in [0, \infty).$ Since 
$\sup \{ \dfrac{ \alpha(2s) }{\alpha(s)} \; | \;s\in[0,\infty) \} = 4$ is bounded, 
there exists a positive number $\beta$ such that
$$
|||a+b|||_{\alpha, p} \le \beta (|||a|||_{\alpha, p} + |||b|||_{\alpha, p}). 
$$
for any $a,b \in \M$ with $\tau({\rm supp}(a)),\tau({\rm supp}(b))<\infty$.
Since $\alpha$ is not concave, $||| . |||_{\alpha, p}$ is not a norm for any $p \ge 1$. 
\end{example}

\begin{example}\rm 
Let $\M$ be a factor of type ${\rm II}_\infty$ and fix $p \ge 1$. Put $\alpha(s) = s^{1/2}$ for $s \in [0, \infty).$ Since  
$\alpha$  is concave, 
$$
|||a+b|||_{\alpha, p} \le |||a|||_{\alpha, p} + |||b|||_{\alpha, p}. 
$$
for any $a,b \in \M$ with $\tau({\rm supp}(a)),\tau({\rm supp}(b))<\infty$.
\end{example}

Under these considerations, we shall show that non-linear traces of the Choquet type are closely related to Lorentz spaces if $\alpha$ is {\it concave}. 
We recall the definition of Lorentz function spaces as in \cite[p.107]{krein-petunin-semenov} and 
Lorentz operator spaces as in \cite[p.65]{L-S-Z}. 

%%%%%%%%%%%%%%%%%%%%%%%%%%%%%%%
%     Definition 2.30
%%%%%%%%%%%%%%%%%%%%%%%%%%%%%%%
\begin{definition} \rm
Let $\psi \not= 0$ be an increasing concave function on $[0,\infty)$ with $\psi(0) = 0$, and  $b$ be a non-negative decreasing 
function on $(0, \infty)$.  We consider the integral  
$$
\int_0^{\infty} b(t) \ d\psi(t) := \psi(+0)b(+0) + \int_{+0}^{\infty} b(t) \ d\psi(t),  
$$ 
where the second term is an improper Stieltjies integral or which is the same, 
$$
\int_0^{\infty} b(t) \ d\psi(t) = \psi(+0)b(+0) + \int_{0}^{\infty} b(t) \psi'(t)dt,  
$$ 
where the second term is a Lebesgue integral. 
The set $\Lambda_{\psi}$ of all measurable functions $f$ on $(0, \infty)$ for which 
$$
\| f \| _{\psi} :=\int_0^{\infty} f^*(t) \ d\psi(t) < \infty
$$ 
is called  a Lorentz function space and a Banach space with the norm $\| f \| _{\psi}$, 
where  $f^*$ is the decreasing rearrangement of $f$. 

Let $M$ be a semifinite factor with a faithful normal semifinite trace $\tau$. 
The set $\M_{\psi}^{\times}$ of all $\tau$-measurable operators $a  \in  \tilde{\M}$ for which 
$$
\| a \| _{\psi} :=\int_0^{\infty} \mu_t(a) d\nu_{\psi}(t)  < \infty 
$$ 
is called a Lorentz operator space and it is a Banach space with the norm $ \| a \| _{\psi}$,
where  $ \mu_t(a)$ is the $t$-th generalized singular value of  $a$. 
\end{definition}

Lorentz function spaces are typical examples of symmetric function spaces and 
Lorentz operator spaces are typical examples of symmetric operator spaces.  
In \cite{K-S}  using uniform submajorization, Kalton and Sukochev established a 
one-to-one correspondence between the symmetric function spaces and symmetric operator spaces. 
We show that non-linear traces of the Choquet type are closely related to Lorentz operator spaces if 
monotone-increasing left continuous functions $\alpha$ are concave.  But we need to extend the 
equation in Proposition \ref{prop:Stieltjes} from bounded operators to $\tau$-measurable operators.  

%%%%%%%%%%%%%%%%%%%%%%%%%%%%%%%
%     Proposition 2.31
%%%%%%%%%%%%%%%%%%%%%%%%%%%%%%%
\begin{proposition}
\label{prop:lorentz}
Let $\M$ be a factor of type ${\rm II}_{\infty}$. 
Let $\alpha: [0,\infty)  \rightarrow [0, \infty)$ be a 
monotone increasing left continuous function with $\alpha(0) = 0$.  
Let  $\varphi_{\alpha} : \M^+ \rightarrow  [0, \infty]$ be a non-linear trace of the Choquet type associated with $\alpha$ .
Assume that $\alpha$ is concave. 
Then  
for any general $\tau$-measurable operator $a \in \tilde{\M}$, we have that 
$$
\| a \|_{\alpha} = \varphi_{\alpha}(|a|) 
: = \int_0^{\infty}  \alpha(\tau(e_{(s,\infty)}(|a|)))ds = \int_0^{\infty} \mu_t(a) d\nu_{\alpha}(t). 
$$
where we admit $+\infty$.  Moreover,  
the set of all $\tau$-measurable operators $a  \in \tilde{\M}$ for which 
$$
\| a \| _{\alpha} :=  \int_0^{\infty}  \alpha(\tau(e_{(s,\infty)}(|a|)))ds < \infty 
$$ 
coincides with the Lorentz operator space $\M_{\alpha}^{\times}$  and a Banach space with the norm $ \| a \|_{\alpha}$.
\end{proposition}
\begin{proof}
For any $\tau$-measurable operator $a \in \tilde{\M}$, put $e_n = e_{[0,n]}(|a|)$.  Then 
$|a|e_n$ is bounded and 
$$
\mu_t(|a|e_n) \nearrow \mu_t(|a|), 
$$ 
as in Fack-Kosaki  \cite[p.278]{F-K} and 
$$
\int_0^{\infty} \mu_t(|a|e_n) d\nu_{\alpha}(t) \nearrow \int_0^{\infty} \mu_t(|a|) d\nu_{\alpha}(t)
$$
as $n \rightarrow \infty$
by the monotone convergence theorem.  Since 
$$
e_{(s,\infty)}(|a|e_n) 
=  \begin{cases}    
e_{(s,\infty)}(|a|)  & \text{ if } s < n \\
0  &  \text{ if }  n \leq s 
\end{cases}
$$
and  $\alpha(\tau(0)) = 0$, 
$$
\int_0^{\infty}  \alpha(\tau(e_{(s,\infty)}(|a|e_n)))ds = \int_0^{n}  \alpha(\tau(e_{(s,\infty)}(|a|)))ds
 \nearrow  \int_0^{\infty}  \alpha(\tau(e_{(s,\infty)}(|a|)))ds 
$$
as $n \rightarrow \infty$.   Therefore the general case is reduced to the bounded case. Now the rest is clear. 
\end{proof}

%%%
%%%  section 3
%%%
\section{Non-linear traces of the Sugeno type}
In this section,  we study non-linear traces of the Sugeno type on semifinite factors. 
We recall the Sugeno integral with respect to a monotone measure.

%%%%%%%%%%%%%%%%%%%%%%%%%%%%%%%%%%%
%     Definition 3.1
%%%%%%%%%%%%%%%%%%%%%%%%%%%%%%%%%%%
\begin{definition} \rm (Sugeno integral)
Let $\Omega$ be a set and ${\mathcal B}$ a ring of sets on $\Omega$.  
Let  $\mu: {\mathcal B} \rightarrow [0, \infty]$ be a monotone measure on $\Omega$. 
Let $f$ be a non-negative measurable function on $\Omega$.  Then the Sugeno integral of $f$ is defined by 
\begin{align*}
{\rm (S)} \int f d\mu & := \sup_{s \in[0,\infty)} s \land  \mu ( \{ x \in \Omega \ | \ f(x) \geq s \})  \\
& =   \sup_{s \in[0,\infty)} s \land  \mu ( \{ x \in \Omega \ | \ f(x) > s \}) . 
\end{align*}
\end{definition}

\begin{remark} \rm  If $f$ is a simple function with 
$$
f = \sum_{i=1}^{n} a_i \chi _{A_i}  \ \ \ ( A_i \cap A_j = \emptyset  \  ( i\not=j) ), 
$$
then the Sugeno integral of $f$ is given by 

$$
{\rm (S)}\int f d\mu = \lor_{i=1}^{n} (a_{\sigma(i) } \land 
\mu(A_{\sigma(1)} \cup A_{\sigma(2)} \cup \dots \cup A_{\sigma(i)} ) ), 
$$
where $\sigma$ is a permutation on $\{1,2,\dots n\}$ such that 
$a_{\sigma(1)} \geq a_{\sigma(2)} \geq  \dots \geq a_{\sigma(n)}$ . 
 \end{remark}

The Sugeno integral has the following properties: \\
 (1) (monotonicity)  
For any  non-negative measurable functions $f$ and $g$  on $\Omega$, 

$$
0 \leq f \leq g \Rightarrow 0 \leq {\rm (S)} \int f d \mu \leq   {\rm (S)}\int g d \mu .
$$ 
\noindent
(2) (comonotonic F-additivity) 
For any  non-negative measurable functions $f$ and $g$  on $\Omega$, 
if $f$ and $g$ are comonotone, then 
$$
{\rm (S)} \int (f \lor g)  d \mu
= {\rm (S)} \int f d \mu \lor {\rm (S)}\int g d \mu .
$$  

\noindent
(3) (positive F-homogeneity)
For any  non-negative measurable function $f$ on $\Omega$ and 
any scalar $k \geq 0$, 
$$
 {\rm (S)} \int k \land f d \mu = k \land  {\rm (S)} \int f d \mu, 
$$
where $0 \land \infty = 0$.

\begin{remark} \rm 
It is important to note that the Sugeno integral can be {\it essentially} characterized as a non-linear monotone positive functional 
which is positively F-homogeneous and comonotone F-additive.  
\end{remark}

%%%%%%%%%%%%%%%%%%%%%%%%%%%%%%%
%     Definition 
%%%%%%%%%%%%%%%%%%%%%%%%%%%%%%%
\begin{definition} \rm
Let $\M$ be a factor of type  $II_{\infty}$. 
Let $\alpha: [0,\infty)  \rightarrow [0, \infty)$ be a 
monotone increasing function with $\alpha(0) = 0$. Put $\alpha(\infty) = \lim_{t \to \infty}\alpha(t)$ by 
convention.  
Define a non-linear trace $\psi_{\alpha} : \M^+ \rightarrow  [0, \infty)$ of the Sugeno type associated with $\alpha$ 
as follows: For $a \in \M^+$ with the spectral decomposition $a = \int_0^{\infty} \lambda d e_{\lambda}(a)$,  let  
$$
\psi_{\alpha}(a) : = \sup_{s \in [0,\infty)} s \land \alpha(\tau(e_{(s,\infty)}(a))). 
$$
We usually assume that $\alpha$ is continuous.  

Similarly we define  a non-linear trace $\psi_{\alpha}$ of the Sugeno type 
for a factor of type $II_1$ with  $\alpha: [0,1]  \rightarrow [0, \infty)$ without saying any more. 

We should note that 
$$
\psi_{\alpha}(a) 
= \sup_{s \in [0,\infty)} s \land \alpha(\tau(e_{[s,\infty)}(a))). 
$$

In fact, since $\tau$ and $\alpha$ are monotone increasing, 
$$
A :=  \sup_{s \in [0,\infty)} s \land \alpha(\tau(e_{(s,\infty)}(a)))
\leq \sup_{s \in [0,\infty)} s \land \alpha(\tau(e_{[s,\infty)}(a))). 
$$
Conversely, we show that for any $s \in [0,\infty)$ 
$$
s \land \alpha(\tau(e_{[s,\infty)}(a))) \leq A. 
$$
It is trivial if $s = 0$.  Take any $s >0$. For any $\varepsilon >0$ with $0 < \varepsilon < s$, there exists 
$s' >0$ such that $ s -  \varepsilon < s' <s$.  Then 
\begin{align*}
& s \land \alpha(\tau(e_{[s,\infty)}(a))) \leq (s' + \varepsilon) \land \alpha(\tau(e_{(s',\infty)}(a))) \\
& \leq (s' \land \alpha(\tau(e_{(s',\infty)}(a))) ) + \varepsilon \leq A +  \varepsilon. 
\end{align*}
Taking $\varepsilon \to 0$, we get the desired fact. 

Moreover,  for any unitary $u \in \M$, 
$$
\psi_{\alpha}(uau^*) =  \psi_{\alpha}(a),  
$$
that is,  $\psi_{\alpha}$ is unitarily invariant.

For example,  if $a$ has a finite spectrum and 
$a = \sum_{i=1}^{n} a_ip_i$ is the spectral decomposition with $a_1 \geq a_2 \geq \dots \geq a_n > 0$, then 
$$
\psi_{\alpha}(a) =  \lor_{i=1} ^{n}  (a_i  \land  \alpha(\tau(p_1 + \dots + p_i)) ). 
$$
In particular, for any projection $p \in \M$ and any positive scalar $c$, we have that 
$$
\psi_{\alpha}(cp)  = c \land  \alpha(\tau(p) ).  
$$
\end{definition}

We can express non-linear traces of the Sugeno type using the ``generalized $t$-th eigenvalues'' $ \lambda_t(a)$.  
We remark that for any projection $p \in \M$ and a scalar $c \geq 0$,  we have that 
$\lambda_t(cp) = c \chi_{[0,\tau(p))}(t)$  for $t \in [0,\infty)$.  Hence 
$$
\sup_{t \in [0,\infty)} \lambda_t(cp) \land \alpha(t)  = c \land  \alpha(\tau(p) )
$$
if $\alpha$ is continuous. 

\begin{proposition} Let $\M$ be a factor of Type $II_{\infty}$.  
Let $\alpha: [0,\infty)  \rightarrow [0, \infty)$ be a 
monotone increasing continuous function with $\alpha(0) = 0$.   
Let  $\psi_{\alpha} : \M^+ \rightarrow  [0, \infty)$ be a non-linear trace of the Sugeno type associated with $\alpha$ . 
Then the non-linear trace $\psi_{\alpha}$ is also described as follows:
For any $a \in \M^+$, we have that  
$$
\psi_{\alpha}(a) := \sup_{s \in [0,\infty)} s \land \alpha(\tau(e_{(s,\infty)}(a))) 
= \sup_{t \in [0,\infty)} \lambda_t(a) \land \alpha(t) \leq \lambda_0(a) = \|a\|. 
$$
Moreover,  $\psi_{\alpha}$ is lower semi-continuous on  $\M^+$ in the operator norm.
\end{proposition}
\begin{proof}
Let $a \in \M^+$. Recall that 
$$
 \lambda_t(a) := \inf \{s \geq 0 |\ \tau(e_{(s, \infty)}(a)) \leq t \} = \min \{s \geq 0 |\ \tau(e_{(s, \infty)}(a)) \leq t \}. 
$$
Let $A = \sup_{s \in [0,\infty)} s \land \alpha(\tau(e_{(s,\infty)}(a))) $ and 
$B =  \sup_{t \in [0,\infty)} \lambda_t(a) \land \alpha(t)$.\\
($A \geq B$): Take any $t \in [0,\infty)$. 
If  $\lambda_t(a) = 0$, then  $A \geq 0 = \lambda_t(a) \land \alpha(t)$. So we may assume that 
$\lambda_t(a) \not= 0$. For any $\varepsilon$ with $0 < \varepsilon < \lambda_t(a)$, 
there exists $s'$ such that 
$$
\lambda_t(a) - \varepsilon < s' < \lambda_t(a). 
$$
By the definition of  $\lambda_t(a)$, we have that 
$$
\tau(e_{(s', \infty)}(a)) > t.
$$
Since $\alpha$ is increasing, $\alpha(\tau(e_{(s', \infty)}(a))) \geq \alpha(t)$. Hence 
$$
A \geq s' \land \alpha(\tau(e_{(s', \infty)}(a))) \geq (\lambda_t(a) - \varepsilon) \land \alpha(t) 
\geq \lambda_t(a)  \land \alpha(t) - \varepsilon. 
$$
Letting $\varepsilon \to 0$, we have that $A \geq \lambda_t(a)  \land \alpha(t)$. This implies that $A \geq B$. \\
($A \leq B$): Take any $s \in [0,\infty)$. 
If $s = 0$ or $\tau(e_{(s,\infty)}(a)) = 0$,  Then $s \land \alpha(\tau(e_{(s,\infty)}(a))) = 0 \leq B$.  
So we may assume that $s \not= 0$ and $\tau(e_{(s,\infty)}(a)) \not= 0$. 
For any $\varepsilon$ with $0 < \varepsilon < \tau(e_{(s,\infty)}(a))$, 
put $t := \tau(e_{(s,\infty)}(a)) -  \varepsilon > 0$.  Since 
$\tau(e_{(s,\infty)}(a)) = t + \varepsilon > t$, 
by the definition of  $\lambda_t(a)$, we have that $s < \lambda_t(a)$.  
Therefore 
$$
s \land \alpha(\tau(e_{(s,\infty)}(a)))  \leq \lambda_t(a) \land \alpha(t + \varepsilon). 
$$
Since $\alpha$ is continuous on $[0,  \infty)$, letting $\varepsilon \to 0$, we have that 
$$
s \land \alpha(\tau(e_{(s,\infty)}(a)))  \leq \lambda_t(a) \land \alpha(t) \leq B. 
$$
This implies that $A \leq B$. 

Moreover, since   $\M^+ \ni a \rightarrow \lambda_t(a)$ is operator norm continuous, 
$\psi_{\alpha} : a \mapsto  \sup_{t \in [0,\infty)} \lambda_t(a) \land \alpha(t)$ is lower semi-continuous on  $\M^+$ in the operator norm.

\end{proof}

%%%%%%%%%%%%%%%%%%%%%
%% Definition 
%%%%%%%%%%%%%%%%%%%%%
\begin{definition} \rm 
Let $\M$ be a factor of type ${\rm II}_{\infty}$ with a trace $\tau$ or a factor of type  ${\rm II}_{1}$ with a normalized trace $\tau$. 
Let $\psi : \M^+ \rightarrow  {\mathbb C}^+$  be a non-linear positive map.
\begin{itemize}
  \item $\psi$ is {\it positively F-homogeneous} if 
$\psi(kI \land a) = k \land \psi(a)$ for any $a \in  M^+$ and any scalar $k \geq 0$.
%  \item $\psi$ is {\it comonotonic F-additive on the spectrum} if 
%$$
%\psi((f\lor g) (a)) = \psi(f(a)) \lor \psi(g(a)) 
%$$  
%for any $a \in  M^+$ and 
%any comonotonic functions $f$ and $g$   in $C(\sigma(a))^+$, where 
%$f(a)$ is a functional calculus of $a$ by $f$. 
  \item $\psi$ is {\it monotonic increasing F-additive on the spectrum} if 
$$
\psi((f \lor g)(a)) = \psi(f(a)) \lor  \psi(g(a))
$$  
for any $a \in  M^+$ and 
any monotone increasing functions $f$ and $g$  in $C(\sigma(a))^+$, 
where $f(a)$ is a functional calculus of $a$ by $f$. 
Then by induction, we also have 
\[   \psi((\bigvee_{i=1}^n f_i)(a)) = \bigvee_{i=1}^n \psi(f_i(a))  \]
%$$
%\psi(f_1(a) \lor  f_2(a) \lor  \dots \lor f_n(a)) = \psi(f_1(a)) \lor  \psi(f_2(a)) \lor  \dots \lor \psi(f_n(a))
%$$
for any monotone increasing functions $f_1, f_2, \dots, f_n$ in  $C(\sigma(a))^+$ for $i=1,2,\dots,n$.
\end{itemize}
\end{definition}

We characterize non-linear traces of the Sugeno type on  $\M^+$. 
%%%%%%%%%%%%%%%%%%%%%%%%%%%%%
%     Theorem 3.8
%%%%%%%%%%%%%%%%%%%%%%%%%%%%%
\begin{theorem} 
Let $\M$ be a factor of type ${\rm II}_{\infty}$.  
Let $\psi : \M^+ \rightarrow  {\mathbb C}^+$ be a non-linear 
positive map.  Then the following are equivalent:
\begin{enumerate}
\item[$(1)$]  $\psi$ is a non-linear trace $\psi = \psi_{\alpha}$  of the Sugeno type associated with  
a monotone increasing continuous function $\alpha: [0, \infty)\rightarrow [0, \infty)$  with $\alpha(0) = 0$. 
\item[$(2)$] $\psi$ is monotonic increasing F-additive on the spectrum, unitarily invariant, monotone,  
positively F-homogeneous, lower semi-continuous on $\M^+$ in the operator norm 
and for any finite projection $p \in \M$, 
$$\lim_{c\rightarrow \infty} \psi(cp) < +\infty .$$ 
Moreover if we put $\alpha: [0, \infty)\rightarrow [0, \infty)$  by $\alpha (t) = \lim_{c\rightarrow \infty} \psi(cp)$ 
for $t = \tau(p)$ with some finite projection $p \in \M$,  then $\alpha$ is continuous.  
\end{enumerate}
\end{theorem}
\begin{proof}
(1)$\Rightarrow$(2) Assume that $\psi$ is a non-linear trace $\psi = \psi_{\alpha}$  of the Sugeno type associated with  
a monotone increasing continuous function $\alpha$ as in (1).  Let $a$ be in $\M^+$.  
For any monotone increasing functions $f$ and $g$  in $C(\sigma(a))^+$, $f \lor g$ is also monotone increasing.  
Therefore 
$\lambda_t(f(a)) = f(\lambda_t(a))$, $\lambda_t(g(a)) = g(\lambda_t(a))$ and 
$$
\lambda_t((f\lor g)(a)) =(f\lor g)(\lambda_t(a)) =  f(\lambda_t(a)) \lor g(\lambda_t(a)).  
$$
Hence 
\begin{align*}
& \psi_{\alpha}((f \lor g)(a))  
= \lor_{t \in [0,\infty)} (\lambda_t((f \lor g)(a)) \land \alpha(t) )\\
= & \lor_{t \in [0,\infty)} ((f(\lambda_t(a)) \lor g(\lambda_t(a))) \land \alpha(t))\\
= &(\lor_{t \in [0,\infty)} f(\lambda_t(a)) \land \alpha(t)) 
    \lor (\lor_{t \in [0,\infty)} g(\lambda_t(a)) \land \alpha(t))\\
 = & \psi_{\alpha}(f(a)) \lor  \psi_{\alpha}(g(a)) . 
\end{align*}
Thus $\psi_{\alpha}$ is monotonic increasing F-additive on the spectrum. 

For $a,b \in M^+$, assume that  $a \leq b$. Then $\lambda_t(a) \leq \lambda_t(b)$.  
Hence 
\begin{align*}
\psi_{\alpha}(a) & =  \sup_{t \in [0,\infty)} \lambda_t(a) \land \alpha(t) \\
& \leq  \sup_{t \in [0,\infty)} \lambda_t(b) \land \alpha(t) = \psi_{\alpha}(b). 
\end{align*}
Thus $\psi_{\alpha}$ is a monotone map. 

For a positive scalar $k$ and a function $f$ on $[0,\infty)$ such that $f(t) = t$,  we have that 
$kI \land f$ is a monotonic increasing continuous function on  $[0,\infty)$. Hence for $ t \in [0,\infty)$, 
$$
\lambda_t(kI \land a)=  \lambda_t(kI \land f) (a)= (kI \land f) (\lambda_t(a)) = k \land  (\lambda_t(a)).
$$
Therefore 
\begin{align*}
\psi_{\alpha}(k \land a) &=
\sup_{t \in [0,\infty)} \lambda_t(kI \land a) \land \alpha(t)  \\
&= \lor_{t \in [0,\infty)} ((k \land  \lambda_t(a)) \land \alpha(t) ) \\
&=  k \land (\lor_{t \in [0,\infty)} ( \lambda_t(a) \land \alpha(t) ))
= k  \land \psi_{\alpha}(a). 
\end{align*}
Thus $\psi_{\alpha}$ is positively F-homogeneous. 
It is clear that $\psi_{\alpha}$ is unitarily invariant by definition. 
 We already showed that $\psi_{\alpha}$ is 
lower semi-continuous on $\M^+$ in the operator norm. 

Let $p$ be a finite projection in $\M$.
Then 
$$
\lim_{c\rightarrow \infty} \psi_{\alpha}(cp) =  \lim_{c\rightarrow \infty}  c \land  \alpha(\tau(p) )
= \alpha (\tau(p)) < +\infty.
$$

(2)$\Rightarrow$(1) 
Assume that $\psi$ satisfies the condition (2).  
Since  $\psi $ is unitarily invariant, $\alpha$ does not depend 
on the choice of such a projection. Since  $\psi$ is monotone, $\alpha$ is monotone increasing. 
Since  $\psi$ is positively F-homogeneous, putting $c = 0$ and $a = 0$, we have that 
$\psi(0 \land 0) = 0 \land \psi(0) = 0 $.  
This means that $\psi(0) = 0$,  so that $\alpha(0) = 0$.  

Suppose that $b \in \M^+$ has a finite spectrum and 
$b= \sum_{i=1}^{n} b_iq_i$ is the spectral decomposition with $b_1 \geq b_2 \geq \dots \geq b_n > 0$. 
Define  functions $f, f_1, f_2, \dots, f_n \in C(\sigma(b))^+$  by $f(x) =x$ and for $i = 1,2,\dots,n$
$$
f_i = b_i \chi_{\{b_1, b_2,\dots, b_i\}} = b_i I \land 
c \chi_{\{b_1, b_2, \dots, b_i\}}
$$
for $c \geq b_1 = \|b\|$, which does not depend on such $c$. 
Each $f_i$ is a monotone increasing function on $\sigma(b)$. Since 
$f = \lor_{i=1} ^{n} f_i$, we have that 
$
b = (\lor_{i=1} ^{n} f_i)(b). 
$
Since $\psi$ is monotonic increasing F-additive on the spectrum and positively F-homogeneous, 
we have that 
\begin{align*}
&  \psi(b) = \psi((\lor_{i=1} ^{n} f_i)(b))  = \lor_{i=1} ^{n}\psi(f_i(b)) \\
= &  \lor_{i=1} ^{n}\psi(b_iI \land 
c \chi_{\{b_1, b_2, \dots, b_i\}})(b)) \\
= &   \lor_{i=1} ^{n}(b_i \land 
\psi((c \chi_{\{b_1, b_2,\dots, b_i\}})(b)) \\
= &  \lor_{i=1} ^{n} (b_i \land \psi(c(q_1 + q_2 + \dots + q_i)) \\
= &  \lor_{i=1} ^{n} (b_i \land (\lim_{c\to \infty}\psi(c(q_1 + q_2 + \dots + q_i)) \\
= &  \lor_{i=1} ^{n} (b_i \land \alpha(\tau(q_1 + \dots + q_i))) = \psi_{\alpha}(b).
\end{align*}
Next, 
for any positive operator $a \in \M^+$, there exists an increasing sequence  $(b_n)_n$ in $\M$ with  $0 \leq b_n \leq a$ 
such that $(b_n)_n$ converges to $a$ in the operator norm topology and each $b_n$ has a finite spectrum.
Since $\psi$ and $\psi_{\alpha}$  are lower semi-continuous on $\M^+$ in the operator norm and 
$\psi(b_n) = \psi_{\alpha}(b_n)$, we conclude that  $ \psi(a) = \psi_{\alpha}(a)$ by taking their limits. 
Therefore $\psi$ is a non-linear trace $\psi_{\alpha}$  of the Sugeno type associated with  
a monotone increasing continuous function $\alpha$. 

\end{proof}

\section{Triangle inequality for non-linear traces of the Sugeno type}

In this section,  we discuss the triangle inequality and a metric for non-linear traces of the Sugeno type. 
The non-linear trace $\psi_{\alpha}$ has the following max type expression: 
\begin{proposition}
\label{prop:max}
Let $\M$ be a factor of type ${\rm II}_{\infty}$.  
Let $\alpha: [0,\infty)  \rightarrow [0, \infty)$ be a 
monotone increasing continuous function with $\alpha(0) = 0$.  
Let  $\psi_{\alpha} : \M^+ \rightarrow  [0, \infty]$ be a non-linear trace of the Sugeno type associated with $\alpha$ . 
Then 
for any $a \in \M^+$, we have that 
\begin{align*}
    & \psi_{\alpha}(a) \\
 = & \sup \{ \lambda \geq 0 \ | \ \exists \text{ a projection $p \in \M$ s.t. } \lambda \leq \alpha(\tau(p)) 
      \text{ and } pap \geq \lambda p  \}.   
\end{align*}
\end{proposition}
\begin{proof}
Take and fix any $a \in \M^+$.  Let $B$ be the right-hand side:
$$
B :=  \sup \{ \lambda \geq 0 \ | \ \exists \text{ a projection  $p \in \M$ s.t.  } \lambda \leq \alpha(\tau(p)) 
\text{ and } pap \geq \lambda p  \}.
$$
Recall that 
$$
\psi_{\alpha}(a) := \sup_{s \in [0,\infty)} s \land \alpha(\tau(e_{(s,\infty)}(a))) 
= \sup_{t \in [0,\infty)} \lambda_t(a) \land \alpha(t). 
$$
For any $s \in [0,\infty)$, put $\lambda = s \land \alpha(\tau(e_{(s,\infty)}(a)))$ and $p = e_{(s,\infty)}(a)$.  
Then $\lambda \leq  \alpha(\tau(e_{(s,\infty)}(a))) = \alpha(\tau(p))$ and $pap \geq sp \geq \lambda p$. 
Hence $\lambda \leq B$.  Since $s$ is arbitrary, $\psi_{\alpha}(a) \leq B$.  

Conversely,  we  show that $\psi_{\alpha}(a) \geq B$.  Take any $\lambda \geq 0$ such that  
there exists a projection $p \in \M$  such that $\lambda \leq \alpha(\tau(p))$ and $pap \geq \lambda p$.  
If  $\lambda = 0$, then it is trivial that $\psi_{\alpha}(a) \geq \lambda$. We may and do assume that 
$\lambda > 0$.  Take an arbitrary $\epsilon > 0$ such that $\tau(p) > \epsilon > 0$.  
Put $t = \tau(p) - \epsilon > 0$.  Then $\lambda_t(p) = 1,$  and 
$$
\lambda_t(a) \geq \lambda_t(pap) \geq \lambda_t(\lambda p) = \lambda \lambda_t(p) = \lambda.   
$$ 
Since $\lambda \leq \alpha(\tau(p)) = \alpha(t +  \epsilon)$ and $\alpha$ is continuous, 
we have that $\lambda \leq \alpha(t)$.  Combining this with that $\lambda \leq \lambda_t(a)$, 
we get that $\lambda \leq \sup_{t \in [0,\infty)} \lambda_t(a) \land \alpha(t) = \psi_{\alpha}(a) $. 
Therefore $B \leq \psi_{\alpha}(a) $.
\end{proof}

The following proposition is similarly proved. 
\begin{proposition}
Let $M$ be a factor of type ${\rm II}_1$ with a faithful normal finite trace with $\tau(1) = 1$.  
Let $\alpha: [0,1] \rightarrow [0, \infty)$ be a 
monotone increasing continuous function with $\alpha(0) = 0$.  
Let  $\psi_{\alpha} : \M^+ \rightarrow  [0, \infty]$ be a non-linear trace of the Sugeno type associated with $\alpha$ . 
Then 
for any $a \in \M^+$, we have that 
\begin{align*}
 & \psi_{\alpha}(a) \\
 =  &  \sup \{ \lambda \geq 0 \ | \ \exists \text{ a projection $p \in M$ s.t. } \lambda \leq \alpha(\tau(p)) 
\text{ and } pap \geq \lambda p  \}.  
\end{align*}
\end{proposition}

In order to prove a triangle inequality for non-linear traces of the Sugeno type, we need a lemma.  
\begin{lemma}
\label{prop:min}
Let $\M$ be a factor of type ${\rm II}_{\infty}$. 
Let $\alpha: [0,\infty)  \rightarrow [0, \infty)$ be a 
monotone increasing continuous function with $\alpha(0) = 0$.  
Let  $\psi_{\alpha} : \M^+ \rightarrow  [0, \infty]$ be a non-linear trace of the Sugeno type associated with $\alpha$ . 
Take any $a \in \M^+$.  Then for any $\epsilon > 0$, there exists a projection $q \in \M$ such that $q \leq {\rm supp}(a)$, 
$\psi_{\alpha}(a) + \epsilon > \alpha(\tau(q))$ and $(I-q)a(I-q) \leq (\psi_{\alpha}(a) + \epsilon)(I-q)$.
\end{lemma}
\begin{proof}
Put $\mu = \psi_{\alpha}(a) + \epsilon > 0$. Define $q = e_{(\mu,\infty)}(a)$. Then 
$$
(I-q)a(I-q) \leq (\psi_{\alpha}(a) + \epsilon)(I-q).
$$
Recall that 
$$
\psi_{\alpha}(a) = \sup_{s \in [0,\infty)} s \land \alpha(\tau(e_{(s,\infty)}(a))), 
$$
$s \mapsto \alpha(\tau(e_{(s,\infty)}(a)))$ is decreasing and $s \mapsto s$ is increasing to the infinite. 
Since $\mu > \psi_{\alpha}(a)$, 
$$
\alpha(\tau(q)) = \alpha(\tau( e_{(\mu,\infty)}(a))) \leq \psi_{\alpha}(a) < \psi_{\alpha}(a) + \epsilon.  
$$
Since the support projection 
${\rm supp}(a) = e_{(0,\infty)}(a)$,  we have that $q \leq {\rm supp}(a)$.  
\end{proof}

We show that if $\alpha$ is concave, then the triangle inequality holds. 

%%%%%%%%%%%%%%%%%%%%%%%%%%%%%
%     Theorem 
%%%%%%%%%%%%%%%%%%%%%%%%%%%%%
\begin{theorem} 
Let $\M$ be a factor of type ${\rm II}_{\infty}$.  
Let $\alpha: [0,\infty)  \rightarrow [0, \infty)$ be a 
monotone increasing continuous function with $\alpha(0) = 0$ and  $\alpha(t) > 0$ for any $t > 0$.  
Let  $\psi_{\alpha} : \M^+ \rightarrow  [0, \infty]$ be a non-linear trace of the Sugeno type associated with $\alpha$.  
Define $||a||_{\alpha}:= \psi_{\alpha}(|a|)$ for $a \in \M$. 
Assume that $\alpha$ is concave.   
Then 
$|| \ ||_{\alpha}$ satisfies the triangle inequality: For any $b,c \in \M$, 
$$
||b + c||_{\alpha} \leq ||b ||_{\alpha} + || c||_{\alpha}.
$$
A similar fact holds for any factor of type ${\rm II}_1$.  
\end{theorem}
\begin{proof}
First, we assume that $b,c \in \M$ are positive.  Put $a = b + c$.  We may assume that $\psi_{\alpha}(a) > 0$.
Take arbitrary $\epsilon >0$ with $\psi_{\alpha}(a) > \epsilon >0$.  
Let $\lambda = \psi_{\alpha}(a) - \epsilon >0$. By the max type theorem, Proposition \ref{prop:max},  there exists a projection $p \in \M$ such that 
$\lambda \leq \alpha(\tau(p))$ and  $pap \geq \lambda p$.  Therefore 
$$
\lambda  + \epsilon= \psi_{\alpha}(a) \geq  \psi_{\alpha}(pap) \geq \psi_{\alpha}(\lambda p) 
= \lambda \land \alpha(\tau(p)) = \lambda.
$$
Put $\mu = \psi_{\alpha}(pbp) + \epsilon >0$.  Then by Lemma \ref{prop:min}
there exists a projection $q \in M$ such that $q \leq {\rm supp}(pbp) \leq p$ , 
$\mu = \psi_{\alpha}(pbp) + \epsilon > \alpha(\tau(q))$ and 
$$
(I-q)pbp(I-q) \leq (\psi_{\alpha}(pbp) + \epsilon)(I-q).
$$
Therefore 
$$
(p-q)pbp(p-q) = p((I-q)pbp(I-q))p \leq (\psi_{\alpha}(pbp) + \epsilon)(p-q) = \mu(p-q).
$$
Since $\alpha$ is concave on $[0,\infty)$, for any $s \geq 0$ and $t \geq  0$, we have that 
$ \alpha(s + t) \leq \alpha(s) + \alpha(t)$.  
Hence 
$$
\alpha(\tau(p)) = \alpha(\tau(p -q) + \tau(q)) \leq \alpha(\tau(p -q)) + \alpha(\tau(q)).
$$
Then 
$$
\lambda - \mu  \leq  \alpha(\tau(p)) - \alpha(\tau(q)) \leq \alpha(\tau(p -q)).
$$
Moreover, 
\begin{align*}
(p-q)pcp(p-q) & = (p-q)pap(p-q) - (p-q)pbp(p-q)  \\ 
& \geq \lambda (p-q)p(p-q) -\mu(p-q) =  (\lambda -\mu)(p-q).
\end{align*}
By the max type theorem, %Proposition \ref{prop:max}, 
we have that $(\lambda -\mu) \leq \psi_{\alpha}(pcp)$, that is, 
$$
\psi_{\alpha}(a) - \epsilon - (\psi_{\alpha}(pbp) + \epsilon) \leq \psi_{\alpha}(pcp). 
$$
This implies that 
$$
\psi_{\alpha}(a) - 2\epsilon \leq \psi_{\alpha}(pbp) + \psi_{\alpha}(pcp) \leq \psi_{\alpha}(b) + \psi_{\alpha}(c).
$$
Since $\epsilon > 0$ can be chosen arbitrarily small, we conclude that 
$$
\psi_{\alpha}(a) \leq \psi_{\alpha}(b) + \psi_{\alpha}(c).
$$
Next, we consider the general case such that $b,c \in M$ are not necessarily positive. 
By a theorem of Akemann-Anderson-Pedersen \cite{A-A-P}, which  generalizes Thompson's Theorem \cite{To}, 
there exist partial isometries $u,v \in M$ such that 
$$
| b + c | \leq u|b|u^* + v|c|v^*.
$$
Therefore we have that 
\begin{align*}
 \psi_{\alpha}(| b + c |) & \leq \psi_{\alpha}(u|b|u^* + v|c|v^*) \\
   & \leq \psi_{\alpha}(u|b|u^*) + \psi_{\alpha}(v|c|v^*) =  \psi_{\alpha}(|b|) + \psi_{\alpha}(|c|).
\end{align*}
This implies the desired triangle inequality.
\end{proof}

%%%%%%%%%%%%%%%%%%%%%%%%%%%%%%%%%
%    Remark
%%%%%%%%%%%%%%%%%%%%%%%%%%%%%%%%%
\begin{remark} \rm
Let $\M$ be a factor of type ${\rm II}_{\infty}$.  
Let $\alpha: [0,\infty)  \rightarrow [0, \infty)$ be a 
monotone increasing continuous function with $\alpha(0) = 0$ and  $\alpha(t) > 0$ for any $t > 0$ .  
Let  $\psi_{\alpha} : M^+ \rightarrow  [0, \infty]$ be a non-linear trace of the Sugeno type associated with $\alpha$ .  
Define $||a||_{\alpha}:= \psi_{\alpha}(|a|)$ for $a \in \M$. Then $||a||_{\alpha} = 0$ if and only if $a = 0$.  
Assume that $\alpha$ is concave.  
Although $||\lambda a||_{\alpha} = |\lambda| ||a||_{\alpha}$ does not hold in general,  
the triangle inequality above implies that 
$M$ with $d(a,b) := ||a -b||_{\alpha}$ is a metric space. Therefore $\M$ becomes a paranormed space, for example, see 
\cite{wilansky}. 
We  extend $\psi_{\alpha}$ to $\M$ as follows: Let $a \in \M$. Consider the decomposition 
$$
a = \frac{1}{2}(a + a^*) + i  \ \frac{1}{2i}(a - a^*) = a_1 -a_2 + i(a_3 -a_4)
$$
with $a_1, a_2, a_3, a_4 \in \M^+$ and $a_1a_2 = a_3a_4 = 0$.  
We define a non-linear trace 
$\psi_{\alpha} : \M \rightarrow  {\mathbb C}$ by 
$$
\psi_{\alpha}(a) := \psi_{\alpha}(a_1) - \psi_{\alpha}(a_2) + i( \psi_{\alpha}(a_3) - \psi_{\alpha}(a_4)),
$$
where we use the same symbol $\psi_{\alpha}$ for the extended non-linear trace. 
\end{remark}

The case of a factor of type ${\rm II}_1$ is more easily proved.

\end{document}